\documentclass[a4paper]{article} 
\usepackage[left=2.3cm,right=2.3cm,top=2.9cm]{geometry}
\usepackage[utf8]{inputenc}
\title{Fluctuations of two-dimensional determinantal processes associated with Berezin--Toeplitz operators}
\author{Alix
Deleporte\thanks{Universit\'e Paris-Saclay, CNRS, Laboratoire de
  Math\'ematiques d'Orsay, 91405, Orsay, France. \\ \href{mailto:alix.deleporte@universite-paris-saclay.fr}{alix.deleporte@universite-paris-saclay.fr}},
Gaultier Lambert\thanks{KTH Royal Institute of Technology, Department
  of Mathematics, 11428, Stockholm, Sweden. -- \href{mailto:glambert@kth.se}{glambert@kth.se}}}
\date{\today}

\usepackage{graphicx}
\usepackage{amssymb}
\usepackage{amsmath}
\usepackage{amsthm}
\usepackage{tikz}
\usepackage{todonotes}
\usepackage{enumitem}

\usepackage{centernot}
\usepackage{mathtools}
\usepackage{ stmaryrd }

\usepackage{hyperref}
\hypersetup{allcolors=black,
pdftitle={Berezin--Toeplitz truncated fermions}}
\numberwithin{equation}{section}
\usepackage{bbm}

\usepackage{titlesec}

\titleformat{\subsection}[runin]
{\normalfont\large\bfseries}{\thesubsection}{1em}{}

\usepackage{braket}
\usepackage{amsmath,amsthm,amssymb}
\usepackage{physics}
\usepackage{mathtools}
\usepackage{bm}
\usepackage{esvect}
\usepackage[english]{babel}
\usepackage{extarrows} 
\usepackage{mathrsfs}
\usepackage{esint}

\usepackage{setspace}
\usepackage{amsmath}

\numberwithin{equation}{section}
\usepackage{bbm}

\usepackage{braket}
\usepackage{amsmath,amsthm,amssymb}
\usepackage{physics}
\usepackage{mathtools}
\usepackage{bm}
\usepackage{esvect}
\usepackage[english]{babel}

\newtheorem{theorem}{Theorem}[section]
\newtheorem{corollary}[theorem]{Corollary}
\newtheorem{lemma}[theorem]{Lemma}
\newtheorem{proposition}[theorem]{Proposition}

\theoremstyle{definition}
\newtheorem{remark}[theorem]{Remark}
\newtheorem{assumption}{Assumptions}

\renewcommand{\d}{\mathrm{d}}
\newcommand{\dist}{\operatorname{dist}}
\newcommand{\supp}{\operatorname{supp}}
\newcommand{\1}{\mathbf{1}}

\newcommand{\g}{\mathrm{g}}

\newcommand{\N}{\mathbb{N}}
\newcommand{\D}{\mathbb{D}}
\newcommand{\E}{\mathbb{E}}
\newcommand{\C}{\mathbb{C}}
\renewcommand{\O}{\mathcal{O}}
\renewcommand{\P}{\mathbb{P}}
\newcommand{\R}{\mathbb{R}}

\newcommand{\X}{\mathbf{X}}
\newcommand{\Z}{\mathbb{Z}}

\newcommand{\Q}{{\rm Q}}
\newcommand{\bgamma}{\boldsymbol{\gamma}}

\begin{document}
\maketitle

\abstract{We consider a new class of determinantal point processes in the complex plane coming from the ground state of free fermions associated with Berezin--Toeplitz operators. These processes generalize the Ginibre ensemble from random matrix theory.
We prove a two-term Szeg\H{o}-type asymptotic expansion for the Laplace transform of smooth linear statistics.
This implies a  law of large numbers and a central limit theorem for
the empirical field. 
The limiting variance includes both contributions from the bulk and
boundary of the droplet. The boundary fluctuations depend on the
Hamiltonian dynamics associated with the underlying operator and,
generally, are not conformally invariant.}

\tableofcontents

\section{Introduction and results}

\subsection{General setting.}
Let $\mathfrak{X}$ be a metric space equipped with a measure $\bgamma$ and, for $\mathcal{N}\in\N$, let $\phi_1,\ldots,\phi_\mathcal{N}$ be an orthonormal family in $L^2(\mathfrak{X},\gamma)$. The \emph{Slater determinant}
\begin{equation} \label{Slater}
\Phi(x_1,\ldots,x_\mathcal{N})=\frac{1}{\sqrt{\mathcal{N}!}}
\det[\phi_i(x_j)]_{i,j=1}^{\mathcal{N}}\in
L^2(\mathfrak{X}^{N},\gamma^{\otimes N})
\end{equation}
is a fundamental example of a fermionic wavefunction, a pure state which encapsulates the Pauli exclusion principle. 
In quantum many-body mechanics, Slater determinants arise as the
ground states of non-interacting fermionic systems, and they also
provide approximations for the ground states of  certain mean-field
interacting systems \cite{benedikter_meanfield_2014}.
A basic property of Slater determinants is their characterisation by the 1-point \emph{reduced density operator} $\Pi$ which is the orthogonal projection from $L^2(\mathfrak{X},\gamma)$ onto $\operatorname{span}\{\phi_1,\ldots,\phi_\mathcal{N}\}$.

\medskip

One can associate to the wavefunction \eqref{Slater} the following probability
density function on $\mathfrak{X}^{\otimes\mathcal{N}}$:
\[
\frac{\d\P_\Pi}{\d\gamma^{\otimes
\mathcal{N}}}(x_1,\cdots,x_{\mathcal{N}})=|\Phi|^2(x_1,\cdots,x_{\mathcal{N}})
=\frac{1}{\mathcal{N}!}\det_{\mathcal{N}\times \mathcal{N}}[\Pi(x_i,x_j)].
\]
This interprets as the \emph{density of probability} of a particle
configuration at a point $(x_1,\ldots,x_{\mathcal{N}})$ of physical space $\mathfrak{X}^{\mathcal{N}}$.
Since these particles are \emph{indistinguishable}, it is appropriate to view $\P$ as an $\mathcal{N}$-point process on  $\mathfrak{X}$. This point process is \emph{determinantal}
and its \emph{correlation kernel} is the integral kernel of the projection $\Pi$. This means that for any $1\leq k\leq \mathcal{N}$ and any $k$-observable $\mathcal O : \mathfrak X^k  \to \R$ continuous and bounded, viewed as a multiplication operator, one has
\begin{equation}\label{eq:determinantal_observables}\begin{aligned}
\frac{N!}{(N-k)!}\langle \Phi | \mathcal{O} |\Phi\rangle  & =  \int \sum_{j_1\neq
\cdots \neq j_k}\mathcal{O}(x_{j_1},\dots,x_{j_k})  \d\P_\Pi(x_1,\ldots,x_{\mathcal{N}})  \\
&=  \int \mathcal{O}(x_1,\dots,x_k) \det_{k\times k}[\Pi(x_i,x_j)] 
\gamma(\d x_1)\cdots\gamma(\d x_k) .
\end{aligned}
\end{equation}
In probabilistic terms, the correlation functions of the point process $\P$ are given, for every $k\le \mathcal{N}$, by $(x_1,\ldots,x_k) \mapsto  \frac1{k!} \det_{k\times k}[\Pi(x_i,x_j)] $, so that the projection $\Pi$ characterizes the law of the point process.  
We refer to  \cite{Borodin11,Johansson06,hough_determinantal_2006} for some general background and different perspectives on determinantal processes. 

\subsection{Ginibre ensemble.}
One fundamental example of such a determinantal process is the (complex) Ginibre ensemble.
Let $N>0$, and let $\mathfrak X = \C$ equipped with the measure $\d \bgamma(z) = \frac 1\pi \d^2z $. Consider the orthogonal family for $L^2(\C,\bgamma)$ given by 
\begin{equation} \label{Gineigfct}
\phi_k(z) = z^k e^{-N|z|^2/2}\sqrt{\tfrac{N^{k+1}}{k!}} \qquad \qquad k\in \N_0.
\end{equation}
Here, the scale parameter $N>0$ controls the particle density.

Given $\mathcal N \in\N$, the point process obtained from $(\phi_0,\cdots,\phi_{\mathcal{N}-1})$ is called the $\mathcal{N}$-Ginibre process. It describes the eigenvalues of a random matrix $\mathbf{A}$ with density $\propto e^{- N \tr(\mathbf{A}^*\mathbf{A}) }$ on $\C^{\mathcal N \times \mathcal N}$ ($\mathbf{A}$ has i.i.d.~complex Gaussian entries). 
This model was introduced by Ginibre in 1965 \cite{ginibre_statistical_1965}, it has a fundamental place within random matrix theory with many applications in probability theory, as well as in quantum many body physics and quantum chaos. We refer to the survey \cite{byun_progress_2025} for an overview of these applications and many results and references on the Ginibre ensemble. 

One can also consider the case $\mathcal N =\infty$: the $\infty$-Ginibre process is defined as the determinantal point process associated with the projection kernel
\begin{equation} \label{oP}
P_N :(z,x) \in \C^2 \mapsto N e^{N(z\bar x -|z|^2/2- |x|^2/2)} = {\textstyle \sum_{k=0}^\infty}\, \phi_k(z)\overline{\phi_k(x)};
\end{equation}
since $P_N$ is locally trace-class one can still make sense of the
second line of \eqref{eq:determinantal_observables}, for every $k\in
\N$, if $\mathcal{O}$ is compactly supported or has good decay
properties at infinity.

The range of $P_N$ is the \emph{Fock--Bargmann space}
\begin{equation} \label{BFspace}
\mathcal{B}_N:=\big\{ u\in L^2(\C,\bgamma);  u(z)= g(z)e^{-N|z|^2/2},  g(z)\text{ is holomorphic}\big\}.
\end{equation}
$\mathcal{B}_N$ is a closed subspace of
$L^2(\C,\bgamma)$.  The $\infty$-Ginibre
process serves as a local (bulk) model for a general class of Bergman
kernels associated with weighted spaces of holomorphic functions,
arising for instance in connection with orthogonal polynomials
\cite{berman_determinantal_2008,berman_direct_2008,ameur_berezin_2010}
as well as geometric quantization on K\"ahler manifolds \cite{zelditch_szego_2000,shiffman_asymptotics_2002,charles_berezin-toeplitz_2003,ma_spin-c_2002}. 
The $\infty$-Ginibre process is the unique stationary determinantal
process on $\C$ with intensity  $\bgamma$ \cite{rider_noise_2007}. It is a fundamental example of negatively correlated, hyperuniform, stationary point process.

The $\mathcal{N}$-Ginibre process, coming from random matrix theory, is a natural truncation and choosing $N=\mathcal{N}$, the intensity of the process is approximately $N \1\{z\in\D\}\bgamma(\d z)$, where $\D=\{z\in\C ; |z|\le 1\}$ is the unit disk, for large $N$. 
Hence, the particles are approximately uniformly distributed in $\D$ and there is an \emph{edge} which influences the fluctuations of the point process. 

\subsection{Determinantal processes associated with Berezin--Toeplitz operators.}  
One can interpret the $\mathcal{N}$-Ginibre ensemble as the
probability density of the
\emph{ground state} of the fermionic $\mathcal{N}$-body problem,
without interactions, and where each particles interacts with the
environment through a self-adjoint operator  
$H_N = P_NVP_N$ on $\mathcal{B}_N$. here, the \emph{potential} $V : \C
\to \R$ is any smooth radial increasing function.  
One often considers $V(z)=|z|^2$, in which case $H_N \phi_k = \frac{k+1}{N}
\phi_k$ for $k\in\N_0$. More concretely, the $\mathcal{N}$-Ginibre
ensemble is the 
determinantal point process associated with the projection $\1\{H_N
\le \mu\}$  on $\mathcal{B}_N$ for $\mathcal N : = \lfloor N \mu \rfloor$. 

From a mathematical physics perspective, $\mathcal{B}_N$ is the \emph{first Landau level} (eigenspace of lowest energy) for a two-dimensional magnetic Laplacian $(-i\hbar \nabla + A)^2$ with constant magnetic field $A(x)= x^\perp$, and $\hbar=N^{-1}$ is the semiclassical parameter.
This quantum system has been introduced by Laughlin to give a mathematical description of the Quantum Hall Effect \cite{laughlin_elementary_1987}.
Then, the operator $P_NVP_N$ can be viewed as a low-energy
approximation (simplified model) of the two-dimensional Hamiltonian $(-i\hbar
\nabla + A)^2+ \hbar V$. The underlying idea is that if $V$ is smooth
and sufficiently small, it acts as  a perturbation and the low-lying
eigenfunctions are close from the first Landau level. In this regime,
one can approximate $-i\hbar
\nabla + A)^2+ \hbar V \approx \hbar P_N(1+V)P_N$ \cite{oblak_anisotropic_2024,bruneau_magnetic_2023}.

$P_NVP_N$ is a \emph{Berezin--Toeplitz operator}
\cite{berezin_general_1975,folland_harmonic_1989,boutet_de_monvel_spectral_1981,bordemann_toeplitz_1994,charles_berezin-toeplitz_2003}. The
functional and spectral properties of these operators, at least if $V$
is sufficiently regular, are closely related to that of
pseudodifferential operators on $L^2(\mathbb{R})$. More precisely, the
\emph{Fourier--Bros--Iagolnitzer transform} (FBI transform), or
Gaussian wave-packet transform, unitarily maps $\mathcal{B}_N$ to
$L^2(\mathbb{R})$, and it conjugates Berezin--Toeplitz operators  with
pseudodifferential operators (within good symbol classes). After this
conjugation, the
semiclassical parameter is $N^{-1}$, so that $N\to +\infty$ is a
semiclassical limit. 
These properties extend to Berezin--Toeplitz quantization on more
general K\"ahler manifolds; for simplicity we state our results in the
case of $\mathbb{C}$ but the techniques we use, and part of the
available results, extend to this more general case. 

The spectral properties of $H_N=P_NVP_N$, in the limit $N\to +\infty$, are therefore linked to the Hamiltonian dynamics of $V$, i.e.~the flow generated by the vector field $\nabla^\perp V$. 
In  two space dimensions,  by the \emph{action-angle theorem}, the  property
that $V$ is preserved along this flow suffices to characterise the
dynamics. This property extends to the quantum case: if $\{V=\mu\}$ is
connected and $\mu$ is a regular energy level for $V$
(i.e.~$\nabla^\perp V \neq 0$ on $\{V=\mu\}$), then one can completely
describe the eigenfunctions of $H_N$ with eigenvalues near $\mu$
\cite{charles_quasimodes_2003}. 
We will rely on this \emph{quantum integrability} to describe the edge
fluctuations of the point process.

The goal of this article is to study the determinantal processes
associated with general Berezin--Toeplitz operators on $\C$, in the
sense that $\Pi=\1\{P_NVP_N\leq \mu\}$, thus generalising the case of
the Ginibre ensemble.
We will work with the following class of potentials. 

\begin{assumption} \label{ass:V}
Let $V:\mathbb{C}\to \mathbb{R}$ be  continuous and let $\mu\in \mathbb{R}$. Suppose that there exists $\delta>0$ such that the following conditions hold:
\begin{itemize}[leftmargin=*] \setlength\itemsep{0em}
\item[a)] $\{V\leq\mu+\delta\}$ is a compact set and $\log(V-\mu)$ is Lipschitz  on $\{V\geq\mu+\delta\}$;
\item[b)]  $V\in C^\infty$ on  $\{V<\mu+\delta\}$;
\item[c)] $\{V=\mu\}$ is a non-trivial, connected smooth curve, on which $\nabla^\perp V \neq 0$.
\end{itemize}
\end{assumption}
      
These assumptions are relatively standard when dealing with
semiclassical spectral asymptotics. Assumption a) guarantees that the
spectrum of $H_N$ below the energy $\mu+\delta$ is discrete with
$\mathcal{N}(\mu)<+\infty$. The additional control of the growth of $V$ at infinity is a technical condition which will be useful  to prove suitable concentration properties for the eigenfunctions; see Section~\ref{sec:dec}. 
Assumption b) allows to use the general semiclassical techniques such as functional calculus, Fourier Integral Operators, etc.
Finally, assumption c) implies that the droplet $\mathcal{D}$ is simply connected (with a $C^\infty$-boundary) and that the flow \eqref{flow} is well-defined.
In contrast, if $\{V=\mu\}$ contains several connected components,
they can resonate, making the analysis more difficult, see
\cite{deleporte_central_2023} for an analogous results on $\R$. As in
\cite{deleporte_central_2023}, we expect that, generically, there are
no resonances in the multi-cut Berezin--Toeplitz ensemble.

Under these hypothesis, our main result, Theorem~\ref{thm:CLT} below,
gives a two-term asymptotic expansion for the Laplace transform of
smooth linear statistics $\X(f)=f(x_1)+\cdots+f(x_{\mathcal{N}})$ where
$f\in C^{\infty}(\C,\R)$.

To describe our results, we rely on the following notations.
\begin{itemize}[leftmargin=*] \setlength\itemsep{0em}
\item $N\in [1,+\infty)$ is a scale parameter (inverse semiclassical
parameter). The asymptotic regime under study is $N\to +\infty$.
\item $H_N:= P_NVP_N$ and $\Pi_N := \1\{P_NVP_N\le\mu\}$ are self-adjoint operators
acting on $L^2(\C,\bgamma)$ with $\d \bgamma(z) = \frac 1\pi \d^2z
$.
\item Under Assumptions \ref{ass:V}, $\Pi_N$ is a finite rank
projector (Lemma \ref{lem:eig}), and we let $\mathcal{N}(\mu)=\tr(\Pi_N)$.
\item The \emph{droplet} is $\mathcal{D} : = \{V <\mu\}$, the
  \emph{edge} is $\{V=\mu\}$.
\item Let $\Xi_V$ be the differential vector field associated with
  $\nabla^\perp V$. We fix $z_0\in \{V=\mu\}$ and parametrise
  $\{V=\mu\}$ along the flow of $\Xi_V$ by setting 
\begin{equation} \label{flow}
z_\theta : =\exp(\theta \tfrac{T(\mu)}{2\pi} \Xi_V )(z_0) , \qquad \theta\in[0,2\pi] , 
\end{equation}
where $T(\mu)$ is the \emph{period} of $\Xi_V$ on $\{V=\mu\}$.
\item $\P_{\Pi}$ denotes the probability law of the determinantal point process $\X$ associated with the projector $\Pi$. 
\item For suitable functions $f: \C \to \R$,  $\X(f)$ denotes the corresponding \emph{linear statistic}. For a determinantal point process, if defined, the Laplace transform of the random variable $\X(f)$  is 
\begin{equation}\label{Laplace}
\E_{\Pi}[e^{{\mathbf X}(f)}]=
\det(1+\Pi(e^f-1)\Pi) = \det_{\Pi}(e^f). 
\end{equation}
The first $\det$ is a Fredholm determinant on  $L^2(\C,\bgamma)$ and we view $e^f$ as a multiplication operator on $L^2(\C,\bgamma)$. 
Here $\mathcal{N}=\operatorname{rank}(\Pi)<\infty$, and the second
$\det$ is a $\mathcal N \times \mathcal N$ determinant computed on the
range of $\Pi$; $\displaystyle \det_{\Pi}(e^f) = \det_{\mathcal N
\times \mathcal N} \big[\langle {\rm u}_k| e^f {\rm
u}_n\rangle\big]$. 
\end{itemize}
\subsection{Fluctuation results.}The  goal of this article is to
investigate the fluctuations of the determinantal process $\P_N$
associated with the fermionic state $\Pi_N$ using linear statistics,
in the limit $N\to\infty$. 
Before stating our main results, we formulate a \emph{law of large
  numbers for $\P_{\Pi_N}$}, it shows that the particles condense
uniformly inside the droplet. 
\begin{proposition}\label{prop:LLN}
Under Assumptions~\ref{ass:V}, one has the pointwise asymptotics as $N\to\infty$:
\[
\Pi_N(x,x) \sim N \1\{V(x)<\mu\}, \qquad \forall x\in\C \setminus \{V= \mu\}. 
\]
For any $\delta>0$, these asymptotics holds uniformly in $ \{ |V
-\mu| \ge \delta\}$, with error $O(e^{-\alpha\sqrt{N}})$ for any $\alpha>0$. 
In particular, as $N\to +\infty$, the particle number satisfies  \[{\mathcal N}(\mu) \simeq N\bgamma(\mathcal{D}) = \frac N\pi |\{V<\mu\}|.\] 

Moreover,  for any function $f:\C \to\R_+$  with $ f \le e^{C|\cdot|}$
for some $C\ge 0$, the rescaled linear statistic converges: 
\[
N^{-1}\X(f) \to \int_{\mathcal{D}} f \d \bgamma \qquad\text{in $L^2$ as }N\to\infty .
\]
\end{proposition}

This Weyl law is  well-known, and admits several improvements concerning the speed of convergence of $\mathcal{N}$; we nevertheless give a proof in Section \ref{sec:decay_kernel}, under our hypotheses, as an intermediate step to obtain decay estimates for the correlation kernel.
In particular, these estimates imply universality of local statistics
for the Berezin--Toeplitz ensembles, meaning that the microscopic
point process converges in distribution to the $\infty$-Ginibre point
process in the bulk; see Proposition~\ref{prop:uni}.

For the  $\mathcal{N}$-Ginibre ensemble (which corresponds to $V(z) =|z|^2$ and $\mu=1$ so that $\mathcal{N}=N$ and $\mathcal{D}=\D$), 
Proposition~\ref{prop:LLN} is known as the \emph{circular law}; this law of large numbers is \emph{universal} for non-hermitian Wigner matrices \cite{bordenave_around_2012}.

Proposition~\ref{prop:LLN} contrasts with the law of large numbers for
normal matrix ensembles, or more generally two-dimensional Coulomb
gases. There, the equilibrium measure is determined by a
variational problem; it is the unique minimizer of the (weighted)
logarithmic energy functional. In particular, it has a density
$\frac1\pi (\Delta V) \1\{\mathcal{D}\}$ 
where $V$ is the external potential and the droplet $\mathcal{D}$ is
given in terms of an obstacle problem coming from potential theory. In
particular, this measure is uniform on the droplet only in the case
where $V$ is \emph{harmonic}, that is, $V(z)$ is quadratic on a
neighborhood of $\mathcal{D}$.  
We refer to the book \cite[Part 1]{serfaty_lectures_2024} for a
description of the macroscopic equilibrium property of general Coulomb
gases and the survey \cite{lewin_coulomb_2022} for an overview of this
field. 

\smallskip

Our main result is a \emph{central limit theorem}
(CLT) for our Berezin--Toeplitz determinantal processes. 

\begin{theorem} \label{thm:CLT}
Let $f: \C \to \R$, with $|f|\le C(1+|\cdot|)$ for some $C\ge 0$,  be $C^\infty$ in a neighborhood of $\mathcal{D}=\{V<\mu\}$.
Under Assumptions~\ref{ass:V}, define the Fourier coefficients of
$f$ along the curve $\{V=\mu\}$ as
\begin{equation} \label{Fcoeff}
\widehat{f}_k = \int_{[0,2\pi]} f(z_\theta) e^{-ik\theta}  \frac{\d\theta}{2\pi} ,\qquad k\in \N,
\end{equation}
and define
\begin{equation} \label{var}
\Sigma^1_{\mathcal{D}}(f) := \tfrac12 \|f\|_{\dot{H}^{\frac12}(\partial\mathcal D)}^2  = \sum_{k\ge 1} k |\widehat{f}_k|^2 \, , \qquad 
\Sigma^2_{\mathcal{D}}(f) :=  \tfrac 12 \|f\|_{\dot{H}^{1}(\mathcal{D})}^2 =
\frac 12 \int_{\mathcal{D}}  |\nabla f|^2 \d\bgamma \,.
\end{equation}
Then, one has the following two-term asymptotics as $N\to\infty$:
\begin{equation} \label{asymp}
\E_{\Pi_N}[e^{\X(f)}]= \exp( \E_{\Pi_N}\X(f)+ \tfrac12 \Sigma(f) + o(1) )
\end{equation}
where $\Sigma(f) = \Sigma^1_{\mathcal{D}}(f)+\Sigma^2_{\mathcal{D}}(f)$.
In particular,
\begin{equation} \label{GinCLT}
\X(f) - \E \X(f) \,\Longrightarrow\, \sqrt{\Sigma(f)}\, \mathcal{G}
\qquad \text{as }N\to +\infty.
\end{equation}
\end{theorem}
By Proposition \ref{prop:LLN}, $\E_{\Pi_N}\X(f)=N\bgamma(f)+o(N)$. In
fact, by the sharp Weyl law, one can prove that
$\E_{\Pi_N}\X(f)=N\bgamma(f)+\O(1)$, but even this improved rate of
convergence is not sufficient to replace the expectation by its limit
in \eqref{asymp}. The remainder term typically oscillates, as the
simple example $f=1$ allows to see: $\E_{\Pi_N}\X(f)=\mathcal{N}\in
\N$, but $N\bgamma(f)\in \R$ evolves continuously with $N$.

On the other hand, if $f$ is supported inside $\mathcal{D}$ one has
$\E_{\Pi_N}\X(f)=N\bgamma(f)+O(e^{-\alpha\sqrt{N}})$ for some $\alpha>0$ (see
Proposition~\ref{prop:decay_kernel}). In fact, in this case, it
suffices to assume that $f$ is $C^1$ to obtain \eqref{asymp}, and the
CLT also holds at arbitrary mesoscopic scales
(Theorem~\ref{prop:cltbulk}).

The stabilisation of the boundary fluctuations to a weighted $H^{\frac
  12}$ seminorm as $N\to +\infty$ is a \emph{strong Szeg\H{o}-type
  limit theorem}.  The description of the 
edge fluctuations is the main step where we need the
Assumptions~\ref{ass:V} b) and c).

We recently proved in \cite{deleporte_central_2023} an analogous result
for one-dimensional free fermions under similar assumptions.  
Theorem~\ref{thm:CLT} has many other counterparts for different types
of ensembles coming from \emph{random matrix theory} and we shortly
review the literature in Section~\ref{sec:relwork}.  

\smallskip

The fluctuations of the $\mathcal{N}$-Ginibre process around the
circular law have been first studied in \cite{rider_noise_2007}. If
$f: \C \to \R$ is $C^1$ in a neighborhood of $\D$ and grow at most
exponentially, then (in distribution) \eqref{GinCLT} holds as $N\to\infty$,
where $\mathcal{G}$ is a (standard) complex Gaussian and where
$\Sigma = \tfrac12 \|f\|_{\dot{H}^{\frac12}(\partial\D)}^2 + \tfrac 12\|f\|_{\dot{H}^{1}(\D)}^2$.
These fluctuations can be interpreted in terms of the \emph{Gaussian
free field} (GFF) which is a universal object in  two-dimensional
random geometry. Namely, according to \cite[Corollary 2]{rider_noise_2007}, the (re-centered) \emph{log characteristic polynomial} of the $N$-Ginibre ensemble converges weakly as $N\to\infty$ to the  two-dimensional Gaussian free field on $\D$  \emph{conditioned to be harmonic in $\hat\C\setminus\overline{\D}$}.
The remarkable fact that linear statistics of the Ginibre ensemble
(and related models) exhibit Gaussian fluctuations without
(re)normalization is a manifestation of the strong
\emph{hyperuniformity} of Coulomb-type gases; see
\cite{fenzl_precise_2022} and \cite{leble_dlr_2024} for recent developments.

In \cite{rider_complex_2007}, Rider and Vir\'ag also studied the fluctuations of the $\infty$-Ginibre process (and other canonical invariant two-dimensional determinantal point processes). More precisely, they proved that for any $f\in L^1\cap H^1(\C,\R)$,
\begin{equation} \label{bulkCLT}
\X(f) -  \int f \d \bgamma \,\Longrightarrow\, \sqrt{\Sigma^2_\C(f)}\, \mathcal{G}. 
\end{equation}
This explains the bulk contribution in \eqref{GinCLT}; the edge
contribution $\Sigma^1_\D(f) $ is a manifestation of the
\emph{truncation} in this model. This noise corresponds (up to a
factor $\frac12$) to the contribution of the fluctuations of linear
statistics for free fermions on $\partial\D$ and it relates to the
\emph{strong Szeg\H{o} limit theorem}; see Section~\ref{sec:relwork}.

\smallskip

\medskip

An important observation is that, because of the boundary term, the
variance functional $\Sigma$ in Theorem~\ref{thm:CLT} is not
\emph{conformally invariant}, in contrast with other models such as
the fluctuations of 
Coulomb gases, eigenvalues of random matrices and dimer models (in the
liquid region). In our situation, the variance depends on the vector
field $\Xi_V$ and not only on the geometry of $\partial\mathcal{D}$,
as already emphasised in \cite{oblak_anisotropic_2024}.

\begin{remark}\label{rk:rad}
In case $V(z) = v(|z|^2)$ is radial with $v:\R^+ \to\R^+$ continuous
and increasing, the (complete) eigenbasis of $H_N$ is $(\phi_k)_{k\ge
  0}$, as given in 
\eqref{Gineigfct}.
Indeed, by rotation invariance and scaling, one has for $k,n\in\N_0$, 
\[
\langle \phi_k, H_N\phi_n\rangle = \int    \overline{\phi_k(z)}  \phi_n(z) V(z) \d\bgamma(z)
= \delta_{k, n} \frac1{n!} \int |z|^{2n} v(|z|^2/N) \d\boldsymbol{g}(z)  =  \delta_{k, n} \lambda_n . 
\]
where $\boldsymbol{g}$ denotes the probability of a standard complex
Gaussian random variable.

Then, the joint probability density function of the determinantal point process associated with $\Pi_N =\1(H_N\le \mu)$ is  given by
\begin{equation*} 
({\mathcal N} !)^{-1} \det[\Pi_N(x_i,x_j)]_{i,j=1}^{\mathcal N} 
= ({\mathcal N} !)^{-1}\big|\det[{\rm u}_{n-1}(x_j)]_{n,j=1}^{\mathcal N}\big|^2
= {\mathcal Z}_N^{-1}  \prod_{1\le i<j \le \mathcal N} |x_i-x_j|^2
\prod_{1\le j \le \mathcal N}  \d\boldsymbol{g}(x_j/\varepsilon_N)
\end{equation*}
where ${\mathcal Z}_N(\mu)$ is a normalization constant and $\varepsilon_N = N^{-1/2}$ is the microscopic scale.
This measure corresponds to the law of the eigenvalues of the
(rescaled) $\mathcal N \times \mathcal N$ (complex) Ginibre ensemble;
$V$ only influences the relationship between the scale parameter $N$,
the number of particles $\mathcal{N}$, and the Fermi energy $\mu$;
following Proposition \ref{prop:LLN} one has, asymptotically,
$\mathcal{N}=N\bgamma(\{V<\mu\})$.

If a radial potential $V$ is such that $\nabla V \neq 0$ on the circle
$\{V=\mu\}$, then $\Xi_V = c_\mu \partial_\theta$ in polar coordinates, so that the flow \eqref{flow} satisfies $z_\theta = z_0 e^{i\frac\theta{2\pi}}$; this is consistent with \eqref{GinCLT} and \eqref{Fcoeff}.
\end{remark}

\subsection{Related central limit theorems.}\label{sec:relwork}
The first and foundational example of two-term asymptotics expansion for $ \det_{\Pi_N}(e^f)$ is the \emph{strong Szeg\H{o} limit theorem}, in which case $\Pi_N$ is the Fourier projection on $\operatorname{span}\{e^{i\frac{k\cdot}{2\pi}}; 0\le k<N\}$ on $L^2(\R/\Z)$. The multiplication by $e^f$ corresponds to a Toeplitz operator in the Fourier basis, so $\det_{\Pi_N}(e^f)$ is a Toeplitz determinant. Then, if $f\in H^{1/2}(\R/\Z)$, it holds as $N\to\infty$, 
\begin{equation} \label{SZ}
\det_{\Pi_N}(e^f) = \exp(N \widehat{f}_0 + \Sigma^1_\D(f) +o(1)) , \qquad\qquad
\Sigma^1_\D(f)=\sum_{k\in\N} k |\widehat{f}_k|^2 . 
\end{equation}
These asymptotics have a long history, they play an important role in mathematical physics (starting with the study of the Ising model, then random matrix theory), the theory of orthogonal polynomials and statistics \cite{deift_toeplitz_2013}.
The determinantal point process associated with $\Pi_N$ is known as the \emph{circular unitary ensemble} (CUE), it corresponds to the eigenvalues of a Haar-distributed random matrix on the unitary group $\mathbb{U}_N$. We refer to \cite{Johansson_98} and \cite{diaconis_linear_2001} for probabilistic proofs of  \eqref{SZ}.

\smallskip

To a large extent, Szeg\H{o}-type asymptotics are \emph{universal}. 
For many one-dimensional (\emph{one-cut}) ensembles of random matrix
type, asymptotic fluctuations of linear statistics are Gaussian and
$H^{1/2}$-correlated, this notably includes orthogonal polynomial
ensembles \cite{breuer_central_2017,lambert_clt_2018} (which can be
directly related to Toeplitz operators), $\beta$-ensembles
\cite{Johansson_98,borodin_gaussian_2017} which are studied through \emph{loop equations} (see \cite{Guionnet_19,angst_sharp_2024} for recent developments), Schur processes which are studied using tools from algebraic combinatorics \cite{moll_random_2017,bufetov_fluctuations_2018}
and Schr\"odinger free fermions  \cite{deleporte_central_2023}.

Such results have also been generalized to a time-dependent CUE (coming from Brownian motion on $\mathbb{U}_N$) in \cite{bourgade_liouville_2022} and to (long-range) one-dimensional Riesz gases in \cite{boursier_optimal_2023}; in this case the fluctuations are described by other fractional Sobolev norms depending on the Riesz parameter. We also note that \cite{bourgade_liouville_2022} also relies on decorrelation estimates for determinantal processes in order to decouple the Laplace transform of local linear statistics.
Recently, two-term asymptotic expansions for the partition function and fluctuation of the  Coulomb gas on a Jordan curve/arc are obtained in \cite{courteaut_partition_2025,courteaut_planar_2025,johansson_coulomb_2023}, by studying Fredholm determinants involving Grunsky operators. In this case, fluctuations are given by a $H^{1/2}$-seminorm weighted by the harmonic measure along the curve. 

\smallskip

For two-dimensional ensembles, in the Ginibre universality class, the fluctuations of linear statistics consists of two independent parts;
a bulk term given by $H^1$-seminorm, and a Szeg\H{o} term
($H^{\frac12}$-seminorm) along the boundary coming from the
truncation. 
This has been discovered in \cite{rider_noise_2007} together with an interpretation in term of the \emph{Gaussian free field}. 
This result has been generalised to \emph{normal matrix models}
(two-dimensional orthogonal polynomial ensembles) in
\cite{ameur_fluctuations_2011,ameur_random_2015}.  
These works rely on asymptotic approximation of the Bergman kernel of
$L^2(\C, e^{-N\varphi}\d\bgamma)$ where $\varphi$ is a suitable
(real-analytic) potential to prove a CLT for the corresponding
determinantal process \cite{berman_determinantal_2013}.

Following
\cite{rider_noise_2007,ameur_fluctuations_2011}, the mainly used technique is the
\emph{cumulant method} to describe the bulk fluctuations. Then, in
\cite{ameur_random_2015}, using the \emph{loop equation method}
inspired from \cite{johansson_szegos_1988}, the authors obtain
two-term asymptotics for the Laplace transform of linear statistics,
including explicit descriptions for the correction term to the mean
and the variance boundary term. 
In particular, the limiting variance has the following analytic
structure:
\[
\hat\Sigma_{\mathcal{D}}(f) =  \hat\Sigma^1_{\mathcal{D}}(f) +\Sigma^2_{\mathcal{D}}(f)  ,
\qquad\qquad
\hat\Sigma^1_{\mathcal{D}}(f) = \frac{1}{2\pi} \Im\bigg( \int_{\partial\mathcal{D}}  f\partial f \d z \bigg),
\]
where the Coulomb droplet $\mathcal{D}$ (associated with the potential) is described in terms of weighted logarithmic potential
theory. In \cite{ameur_random_2015}, it is assumed that $\Delta Q >0$
on $\mathcal{D}$, the boundary $\mathcal{D}$ is an analytic Jordan
curve (one-cut regular condition).  
This variance is \emph{conformally invariant}\footnote{It means that
for any conformal map $\psi : \mathcal{D} \to \D$ and any map $f:\D
\to \R$ (bounded), $C^1$ in a neighborhood of $\D$, one has
$\Sigma_{\mathcal{D}}(f\circ\psi) = \Sigma_\D(f)$. 
} and the CLT can be interpreted as the convergence of the (centred)
logarithmic potential of the point process to the Gaussian free field
(GFF), conditioned to be harmonic outside of $\mathcal{D}$. Hence, the
fluctuations of these two-dimensional orthogonal polynomial ensembles
(Coulomb gas at $\beta=2$) are \emph{universal} in the sense that they
depend only on the potential $V$ through the set $\mathcal{D}$, at
least in the one-cut case.  
In contrast, for Berezin--Toeplitz ensembles, the limiting variance in
Theorem~\ref{thm:CLT} cannot be interpreted in terms of the GFF 
(except in the harmonic case\footnote{Any radial or elliptic
potential.}) because in the boundary term
$\Sigma^1_{\mathcal{D}}(f)$  is not \emph{conformally invariant}: the
Fourier coefficients are computed along the Hamilton flow
\eqref{flow}, which is not invariant under conformal
transformations.

\smallskip

These results (two-term asymptotic expansion for the Laplace transform
of a linear statistic) have been generalised to two-dimensional
Coulomb gases at any fixed $\beta>0$ based on energy methods and
free energy expansions in \cite{leble_large_2017,BBNY19} using two
different approaches which both rely on \emph{decoupling estimates} to
obtain the GFF fluctuations. The CLT from \cite{BBNY19} holds for test
functions supported in the bulk and the proof relies on a quasi-free
Yukawa approximation to obtain the decoupling. 
In contrast, the CLT from \cite{leble_large_2017} is valid for general
test functions and includes the boundary contribution to the
variance. The decoupling estimates are obtained using the
\emph{screening procedure} from \cite{leble_fluctuations_2018}.
Precise quantitative stability estimates on the solutions to the
obstacle problem from potential theory, obtained in
\cite{serfaty_quantitative_2018}, are necessary to derive boundary
contribution to the variance.  
Finally, in the bulk, the CLT has been strenghten to mesoscopic scales
and \emph{high temperature regimes} in \cite{serfaty_gaussian_2023}
based on new local laws valid for general Coulomb gases. The proofs
are summarized in the recent textbook \cite[Part
3]{serfaty_lectures_2024}. These techniques have also been used in
\cite{lambert_law_2024,peilen_maximum_2025} to compute the leading
order asymptotic of the maximum of the two-dimensional Coulomb field
inside the droplet, verifying the prediction coming from the GFF
analogy.
Generalizing in a different direction, central limit theorems (also
valid on mesoscopic scales in the bulk) have been obtained for linear
eigenvalue statistics of large non‐Hermitian Wigner matrices in
\cite{cipolloni_central_2023,cipolloni_mesoscopic_2024} using
multi-resolvent asymptotics.  
Then, the leading order asymptotic of the maximum of the
two-dimensional Coulomb field (log-characteristic polynomial of the
non-Hermitian random matrix) have been obtained in
\cite{cipolloni_maximum_2025}.
The log-characteristic polynomial of large non-Hermitian matrices has
also been recently studied under the Brownian evolution, leading to
the converges in distribution to a 2+1 dimensional logarithmically
correlated (with respect to the parabolic distance) Gaussian field \cite{bourgade_fluctuations_2024}.

Free fermions models have also been studied in higher dimension
\cite{torquato_point_2008}, including determinantal point processes on
complex manifolds
\cite{berman_determinantal_2014,berman_sharp_2012,berman_determinantal_2008,charles_entanglement_2018}. This
adapts the orthogonal polynomial setup to a compact complex manifold
$\mathfrak{X}$ equipped with an orthogonal projection $\Pi$ on
$H_0(\mathfrak{X} , \mathcal{L}^{\otimes N})$ where $\mathcal{L}$ is a
holomorphic line bundle equipped with a Hermitian metric with a
(local) weight $e^{-\phi}$. In this geometric setting, Berman
obtained a large deviation principle \cite{berman_determinantal_2014}
for the determinantal point process associated with $\Pi$ and (on a
compact Riemann surface) a central limit theorem in
\cite{berman_sharp_2012}. The equilibrium measure is the
Monge-Amp\`ere measure on $\mathfrak{X}$ and a strong Szeg\"o-type
limit theorem holds with limiting variance $\phi \mapsto
\|d\phi\|_{\mathfrak{X}}^2$, hence the fluctuations converge weakly to
the GFF on $\mathfrak{X}$. These results are reviewed and generalized
in \cite{berman_determinantal_2008} together with establishing
\emph{microscopic universality}: the kernel converges locally in the
bulk to the Fock space Bergman kernel (Ginibre kernel in dimension
1). In this context, the area law for particle fluctuations and the
entanglement entropy has been derived in
\cite{charles_entanglement_2018}. The technical cornerstone is an
analysis of the kernel of the projector \cite{boutet_de_monvel_sur_1975,shiffman_asymptotics_2002,berman_direct_2008}.

Finally, in the Euclidean setting, local universality and a central
limit theorem in the bulk have been derived in \cite{deleporte_universality_2024} for Schr\"odinger free fermions on $\R ^n$, including a strong Szeg\"o-type limit theorem in dimension $n=1$ in \cite{deleporte_central_2023}.

\subsection{Proof strategy and techniques.}

The proof of Theorem~\ref{thm:CLT} consists in obtaining the asymptotics of the Laplace transform of the linear statistic $\X(f)$ under $\E_{\Pi_N}$ using the determinantal formula \eqref{Laplace}.
Let us introduce the following notation: for a (locally trace-class) projection $\Pi$ and a function $f\in C_0(\C,\R)$, define
\begin{equation} \label{def:Ups}
\Upsilon :(f;\Pi) \mapsto \log\det(1+\Pi (e^{f}-1)\Pi) -  \tr(\Pi f\Pi)  .
\end{equation}
The point is that $\E_{\Pi}[\X(f)]= \tr(\Pi  f\Pi)$, so that $\Upsilon(f;\Pi) = \log \E_{\Pi}[e^{\X(f)-\E_\Pi\X(f)}]$ is
the $\log$ Laplace transform of the centred linear statistics.

\smallskip

In Theorem~\ref{thm:CLT}, the fluctuations of $\X(f)$ decomposes,
asymptotically, into two independent parts:
\begin{itemize}[leftmargin=*] \setlength\itemsep{0em}
\item[-] a bulk term $\Sigma^2_{\mathcal{D}}$ which is \emph{universal} and akin to the $\infty$-Ginibre case restricted to the droplet $\mathcal{D}$, 
\item[-] a boundary term $\Sigma^1_{\mathcal{D}}$ given by a $H^{1/2}$-seminorm weighted by the Hamiltonian flow along the curve $\{V=\mu\}$. 
\end{itemize}
These two parts arise from different asymptotics of the projection kernel $\Pi_N$ in the bulk and at the boundary of the droplet. The proof proceeds in three steps. 
\begin{itemize}[leftmargin=*] \setlength\itemsep{0em}

\item[1]  In Section~\ref{sec:dec}, we obtain a decorrelation estimate
to separate the contribution from \emph{the bulk} and the \emph{the
boundary}. Let $\delta>0$ be a small parameter. To simplify, we show that 
\begin{equation} \label{split}
\Upsilon(f;\Pi_N) = \Upsilon(f_1;\Pi_N) + \Upsilon(f_2;\Pi_N)+\text{overlap} 
\end{equation}
where $f_1$ is supported in $\{|V-\mu| \le \delta\}$ (boundary) and $f_2$ is supported in $\{V\le \mu- \delta/2\}$ (inside the bulk). The contribution from the exterior region $\{V\le \mu+ \delta/2\}$ is exponentially small and the two parts $f_1,f_2$ have an overlap  which lie inside the bulk. The estimates \eqref{split} relies on the fact that, apart from along the boundary of the droplet, the correlation kernel $\Pi_N$ decays rapidly away from the diagonal. This fact is well-known for the Ginibre ensemble and we generalise it to other Berezin-Toeplitz spectral projector using methods from semiclassical analysis.

\item[2] In the bulk, the projection $\Pi_N$ is approximated by $P_N$ (up to exponentially small errors).  This allows to approximate as $N\to\infty$,
\begin{equation} \label{eq:Ups2}
\Upsilon(f_2;\Pi_N)  \sim  \Upsilon(f_2;P_N) \to \tfrac 12 \Sigma^2_\C(f_2) .
\end{equation}
The last asymptotics follow from the CLT for the $\infty$-Ginibre ensemble and it suffices to assume that $f_2 \in C^1$.
Similar asymptotics also holds for the overlap term which contributes to a small part of the bulk variance in the limit.
In this regard, we cannot directly apply the results from \cite{rider_noise_2007,rider_complex_2007,ameur_fluctuations_2011} because the convergence in these papers hold for cumulants and not a priori for log determinants. 
One could however use the results from \cite{ameur_random_2015,leble_fluctuations_2018,BBNY19} which are also valid for more general ensembles with sophisticated proofs (see Section~\ref{sec:relwork}).
Instead, in Section~\ref{sec:bulk}, we give an elementary proof of \eqref{eq:Ups2} based on the determinantal structure. This result can be of independent interest since it is not specific to the Ginibre ensemble. 

\item[3] At the boundary of the droplet, the model is not translation invariant but the asymptotic behavior of the operator $\Pi_N$ is still universal, albeit more involved. 
Using this, in Section~\ref{sec:edge}, we give a proof of the CLT in case the test function $f_1$ is supported inside $\{|V-\mu| \le \delta/2\}$ by comparing the determinant in \eqref{def:Ups} with a Toeplitz determinant.

Let  $\mathcal{X}_N : = \1\{\mu-\delta< H_N \le \mu+\delta\}$ and $\mathcal A := \mathcal{X}_N(e^{f_1}-1)\mathcal{X}_N$. 
Hence, $\mathcal A$  is \emph{localized} in a small energy window around $\mu$.
Let $\vartheta =\log(1+\cdot)$.
We first show that  
\begin{equation} \label{Ups0}
\Upsilon(f_1;\Pi_N) \simeq  \tr\big[\vartheta(\Pi_N\mathcal A \Pi_N) - \Pi_N\vartheta(\mathcal A)\Pi_N \big]  + \tr\big(\Pi_N \big(\vartheta(\mathcal A) - f_1\big)  \Pi_N\big)
\end{equation}
and that the matrix elements $A_{k,n} =\langle {\rm u}_k, \mathcal A
{\rm u}_n\rangle $ in the orthogonal basis of $H_N$ have an \emph{approximately Toeplitz structure}, meaning that $A_{k,n}$ are approximately constant along diagonals if the parameter $\delta$ is small.
This follows from the \emph{integrability} of $H_N$, in the sense that it is possible to obtain asymptotics expansion (to any accuracy) for the eigenvalues of $H_N$; these are instances of \emph{WKB approximation} (Wentzel–Kramers–Brillouin). 

Then, we develop a new \emph{replacement principle} in order to show
that the first term in \eqref{Ups0} can be approximated as
\begin{equation} \label{Ups1}
\tr\big[\vartheta(\Pi_N\mathcal A \Pi_N) - \Pi_N\vartheta(\mathcal A)\Pi_N \big] 
\simeq \tr\big[\vartheta(\Pi  B \Pi ) - \Pi \vartheta(B)\Pi \big].
\end{equation}
The right-hand side is the trace of an operator on $L^2(S^1)$: $B$ is
the operator of multiplication by the symbol $b:\theta \ni
[0,2\pi] \mapsto e^{f_1}(z_\theta)$, where $z_{\theta}$ is given by \eqref{flow}
(that is, in the Fourier modes, $B$ is an infinite Toeplitz matrix
with entries $B_{k,n} = \widehat{b}_{k-n}$ for $k,n \in\Z$);
$\Pi$ is the projection onto $\operatorname{span}({\rm e}_k , k<0)$ in the canonical basis.
At this stage, the potential $V$ comes into play through its
Hamiltonian flow, which encodes the semiclassical asymptotics of the
eigenfunctions $\{{\rm u}_k\}$.
Note that this replacement principle is also not specific to the
Berezin--Toeplitz case and may  also be of independent interest.

Then, applying the strong Szeg\H{o} limit theorem to the RHS of \eqref{Ups1}, one has as $N\to\infty$, 
\[
\tr\big[\vartheta(\Pi_N\mathcal A \Pi_N) - \Pi_N\vartheta(\mathcal A)\Pi_N \big]  \to \tfrac 12 \Sigma^1_{\mathcal{D}}(f_1).
\]
This is the origin of the weighted $H^{1/2}$-seminorm  along the boundary. 
As for the second term in \eqref{Ups0}, one has
$\mathcal A \simeq \mathcal{G} := P_N (e^{f_1}-1) P_N$ up to a small error (controlled in the trace-norm) where $\mathcal{G}$ is a  Berezin-Toeplitz operator with as smooth, bounded, symbol. 
Then, one can use the two term semiclassical expansion
\[
\vartheta(\mathcal A) \simeq \vartheta(\mathcal{G})
\simeq P_N\big(\vartheta (e^{f_1}-1)- N^{-1} \vartheta''(e^{f_1}-1) |\partial(e^{f_1})|^2\big) P_N .
\]
Thus, by Theorem~\ref{prop:LLN} and using that $\vartheta (e^{f_1}-1) = f_1$, $\vartheta''(e^{f_1}-1) = - e^{-2f_1}$ and 
$|\partial(e^{f_1})| = \tfrac12 |\nabla f_1| e^{f_1}$, we obtain as $N\to\infty$, 
\[
\tr\big(\Pi_N \big(\vartheta(\mathcal A) -  f_1\big)  \Pi_N\big) 
\sim \frac1{4N} \tr\big(\Pi_N |\nabla f_1|^2\Pi_N\big) 
\to\frac{1}{4\pi}\int_{\mathcal{D}} |\nabla f_1|^2  =\frac 12 \Sigma^2_{\mathcal{D}}(f_1) . 
\]
\end{itemize}
Altogether, by \eqref{Ups0}, we obtain  
$\Upsilon(f_1;\Pi_N) \to \tfrac 12 \Sigma^1_{\mathcal{D}}(f_1)+ \tfrac 12 \Sigma^2_{\mathcal{D}}(f_1)$ as $N\to\infty$. Then, going back to \eqref{split} and using that $\Sigma^1_{\mathcal{D}}(f_1) = \Sigma^1_{\mathcal{D}}(f)$ and $\Sigma^2_{\mathcal{D}}(f_1)+\Sigma^2_{\mathcal{D}}(f_2) + \Sigma^2_{\mathcal{D}}(\text{overlap}) = \Sigma^2_{\mathcal{D}}(f)$,   we conclude that as $N\to\infty$
\[
\Upsilon(f;\Pi_N) \to \tfrac 12 \Sigma^1_{\mathcal{D}}(f)+ \tfrac 12\Sigma^2_{\mathcal{D}}(f) .
\]

\subsection{Notations.}
In this article, we will use the following notational conventions.
Recall that we work on $L^2(\C,\bgamma)$ equipped with the measure $\d \bgamma(z) = \frac 1\pi \d^2z $.

\begin{itemize}[leftmargin=*] \setlength\itemsep{0em}
\item Let $\Pi$ be a projection operator on $L^2(\C,\bgamma)$. Recall
  that we denote by $\P_{\Pi}$  the  law of the determinantal point
  process associated with the projection operator $\Pi$ and define, for suitable functions $f:\C\to\R$,
\[
\Upsilon(f;\Pi) = \log\det(1+\Pi (e^{f}-1)\Pi) -  \tr(\Pi f\Pi)  =  \log \E_{\Pi}[e^{\X(f)-\E_\Pi\X(f)}] . 
\]

\item Given $A:L^2(\C,\bgamma)\to L^2(\C,\bgamma)$, we denote by
$\|A\|$ the operator norm of $A$ on this space.
In particular, if $f:\C\to\R$ is a bounded function, viewed as a multiplication operator, 
$\|f\| = \|f\|_{C^0} =\sup\{|f(x)|; x\in\C\}$. 

\item If $A$ is a bounded operator, if defined, we denote by $(x,z) \mapsto A(x,z)$ the integral kernel of $A$.
Then, by Schur's test:
\[
\|A\|^2 \le \sup_{x\in\C}\left(\int |A(x,z)| \d\bgamma(z), \int |A(z,x)| \d\bgamma(z) \right).
\]
\item If $A$ is Hilbert-Schmidt, we denote by $\|A\|_{\rm H}$ its Hilbert-Schmidt
norm (Schatten-2 norm):
\[
\|A\|_{\rm H}^2 = \iint |A(x,z)|^2 \d\bgamma(x)  \d\bgamma(z) . 
\]
\item If $A$ is trace-class, we denote by $\|A\|_{\tr}$ its trace norm (Schatten-1 norm). We have
\[
\|A\|_{\tr}  \le \iint |A(x,z)|   \bgamma(\d x) \bgamma(\d z) .
\]
\item Given two functions $a_N, b_N \ge 0$ for $N\ge 1$, we write
$a_N\lesssim b_N$ is there exists constants $C, N_0 \ge 1$ such
that $a_N \le C b_N$ for all $N\ge N_0$.
\end{itemize}

\section{Decoupling the bulk and the boundary}
\label{sec:dec}

Recall the definition \eqref{oP} of the projection $P_N$, $N\ge 1$, acting on $L^2(\C,\bgamma)$ and denote $\varepsilon_N:=N^{-\frac 12}$ for the microscopic scale.
This section aims at giving estimates on the eigenfunctions of $P_NVP_N$ below the energy $\mu+\delta$ under Assumptions~\ref{ass:V} a) and to deduce (pointwise) bounds for the integral kernel of the spectral projection $\Pi_N :=\1(P_NVP_N < \mu)$.

\subsection{Basic estimates.} \label{sec:Best}
The first order of business is to prove that $\Pi_N$ is a finite rank
projector.

\begin{lemma} \label{lem:eig}
Under Assumption~\ref{ass:V} a), the spectrum of $H_N$ below $\mu$ consists of finitely many eigenvalues (counting multiplicities). 
\end{lemma} 
\begin{proof}
Since $\{V\leq\mu+\delta\}$ is compact, let $B$ be an open ball,
centred at 0, such that $\{V\leq\mu+\delta\} \subset B$. Define
\[
V_1=
\begin{cases}
\min(V)&\text{ on }B\\
\mu + \delta&\text{ outside }B.
\end{cases}
\]
Then $V_1\leq V$, so that the associated
spectral projectors satisfy \[\tr(\1(P_NVP_N\leq \mu))\le
  \tr(\1(P_NV_1P_N\leq \mu)).\]

Since $V_1$ is radial, the spectral data of $P_NV_1P_N$ is explicit
(see Remark~\ref{rk:rad}); this operator has simple eigenvalues which  
accumulate at $\mu+\delta$. 
Therefore,
\[
\tr[\1(P_NVP_N\leq \mu)] \le \tr[\1(P_NV_1P_N\leq \mu)] <\infty.
\]
This proves that the spectral projector $\Pi_N =\1(P_NVP_N\leq \mu)$ has finite range.
\end{proof}

Observe that $\Pi_N$ is a projection on a (finite-dimensional) subspace of $\mathcal{B}_N$, so that its kernel $\C^2\ni(x,z) \mapsto \Pi_N(x,z)$ is smooth. 
Since $0\leq \Pi_N\leq P_N$ as operators, one has
\begin{equation} \label{kerptwbd}
0
\leq \Pi_N(x,x)\leq P_N(x,x)= N  , \qquad x\in\C .
\end{equation}
Then, by Cauchy-Schwarz, 
\begin{equation} \label{CSker}
|\Pi_N(x,z)|\leq \sqrt{\Pi_N(z,z)\Pi_N(x,x)} \le N    , \qquad (x,z)\in\C^2. 
\end{equation}
Similarly, the complementary projection satisfies
\begin{equation} \label{CSker2}
|P_N(x,z)-\Pi_N(x,z)|\leq \sqrt{(N-\Pi_N(z,z))(N-\Pi_N(x,x))} \le N    , \qquad (x,z)\in\C^2. 
\end{equation}

We also record that if $u\in \mathcal H_N$, $z\in\C\mapsto |u(z)|^2e^{N|z|^2}$ is subharmonic, as the square modulus of a holomorphic function. Thus, for every $\epsilon>0$,
\[
|u(0)|^2 \le  \frac{1}{\epsilon^2} \int_{|z|\le \epsilon} |u(z)|^2e^{N|z|^2} \d\bgamma(z).
\] 
Choosing $\epsilon = \varepsilon_N = N^{-\frac12}$ and using that $e/\pi \le 1$, we obtain for any $u\in \mathcal{B}_N$,
\begin{equation} \label{uptwbd0}
|u(0)|  \le N^{\frac 12} \|u\1\{\D(0, \varepsilon_N)\}\|_{L^2}.
\end{equation}
One cannot immediately translate this inequality at other points since the space $\mathcal{B}_N$ is not translation-invariant. However one can apply \emph{magnetic translations}: given  $u\in \mathcal{H}_N$ and $x\in \C$,  the function
\[
u_x:z\mapsto u(z-x)\exp(N(z\bar{x}-\bar{z}x))
\]
is in $\mathcal{B}_N$. Therefore, by applying \eqref{uptwbd0} to $u_x$ we obtain for any $x\in\C$, 
\begin{equation} \label{uptwbd}
|u(x)|  \le N^{\frac 12} \|u\1\{\D(x,  \varepsilon_N)\}\|_{L^2}.
\end{equation}
In particular, if $P_Nu=u$ (so that $u\in \mathcal{B}_N$) and  $\|u\|_{L^2}=1$ then $\|u\|_{L^\infty} \le N^{\frac 12}$. This inequality  is true for any (normalised) eigenfunctions of $H_N : = P_NVP_N $.

In the remainder this section, we record some properties of the projection $P_N$. 

\begin{lemma} \label{lem:Hop}
Let $\alpha\ge 0$, let $\rho : \C \to \R$ be a $1$-Lipschitz function,
and given $\alpha\geq 0$ define 
\[P_N^{\alpha} :=
e^{\alpha\rho/\varepsilon_N}P_Ne^{-\alpha\rho/\varepsilon_N}.\] There
exists $C_0$ such that for all $x,z\in\C^2$,
\[
|P_N^{\alpha}(x,z)|\le C_0e^{\alpha^2} N\exp(-N|x-z|^2/4), \qquad 
\|P_N^\alpha-P_N\| \le C_0 \alpha e^{\alpha^2} .
\]
\end{lemma}

\begin{proof}
The kernel of $A=e^{\alpha\rho/\varepsilon_N}P_Ne^{-\alpha\rho/\varepsilon_N}-P_N$ is explicit and, using that $\rho$ is 1-Lipschitz, it satisfies 
\[
|A(x,z)| = N \big|e^{\alpha(\rho(x)-\rho(z))/\varepsilon_N}-1\big|
e^{- N|x-z|^2/2} \le \alpha \varepsilon_N^{-3/2} |x-z| e^{-
  |x-z|^2/2\varepsilon_N^2 + \alpha |x-z|/\varepsilon_N} .
\]
Consequently,
\begin{align*}
\|A\|
\le \alpha\sup_{x\in\C} \int_\C  \varepsilon_N^{-3/2} |x-z| e^{-
  |x-z|^2/2\varepsilon_N^2 + \alpha |x-z|/\varepsilon_N } \bgamma(\d z) 
  & = \alpha \int_{\C} |Z| e^{-|Z|^2/2+\alpha|Z|}\d^2z\\
  &\leq \alpha e^{\alpha^2}\int_{\C}|z|e^{-|Z|^2}{4}\d^2Z
\end{align*}
where we used the fact that
\[
  -\frac{|z|^2}{2}+\alpha|z|\leq \alpha^2-\frac{|z|^2}{4}\qquad \forall
  z\in \C.
\]
This proves the second claim. For the first claim, using the same
inequality, we obtain directly
\[
|P_N^{\alpha}(x,z)| = N e^{\alpha(\rho(x)-\rho(z))/\varepsilon_N}  e^{- N|x-z|^2/2} \le N  e^{\alpha^2} e^{- N|x-z|^2/4}. 
\]
\end{proof}

\begin{lemma} \label{lem:semiclassic_com}
There is a constant $C>0$ so that for any continuous function $g:\C\to\R$, 
\[
\| [P_N,g]\| \le C\, \varepsilon_N \mathrm{Lip}(g) , \qquad\quad \mathrm{Lip}(g) :=\sup_{y\neq x}\left(\frac{|g(y)-g(x)|}{|y-x|}\right) . 
\]
Moreover, if $f :\C\to\R$ is independent of $N$, in $C^1_c$, and
$g_N=f(\cdot \eta_N^{-1})$ for some $\eta_N>0$ with $\eta_N^2 N
\to\infty$, then as $N\to\infty$,  
\[
\| [P_N,g] \|_{\rm H}^2 \to 2\Sigma^2_\C(f) = \int  |\nabla f|^2 \d\bgamma . 
\]
\end{lemma}

\begin{proof} The proof is similar to that of Lemma~\ref{lem:Hop}. The kernel of operator in question is 
\[
[P_N,g] : (x,z) \mapsto (g(z)-g(x))P_N(x,z) .
\] 
Then, by Schur's test,
\[
\| [P_N,g] \| \le  N \mathrm{Lip}(g)\sup_{z\in\C} \int  |x-z| e^{-N|x-z|^2/2} \d\bgamma(x)
= \varepsilon_N \mathrm{Lip}(g) \int_0^\infty \sqrt{t} e^{-t/2} \d t .
\]
This proves the first claim. For the second claim, by a change of
variables, introducing $\tilde{\varepsilon}_N=\varepsilon_N/\eta_N$,
\begin{equation} \label{Gvar}
\| [P_N,g] \|_{\rm H}^2 = \iint \big|[P_N,g]  (x,z)\big|^2 \d\bgamma(x)\d\bgamma(z)
= \frac{1}{\tilde{\varepsilon}_N^2} \iint |f(x)-f(x+Z \tilde{\varepsilon}_N)|^2 e^{-|Z|^2}  \d\bgamma(x)\d\bgamma(Z) .
\end{equation}
Let $\mathcal S \subset \C$ be a  compact set such that  $\operatorname{dist}(\supp f  ; \mathcal S^c) \ge 1$ -- in particular $\supp f \subset \mathcal S$. 
Then
\begin{multline*}
\iint \1\{x\in\mathcal S^c\} |f(x+Z \tilde{\varepsilon}_N)|^2 e^{-|z|^2}  \d\bgamma(x)\d\bgamma(z) \\ \le 
\iint \1\{|z| \ge  \tilde{\varepsilon}_N^{-1}\} |f(x)|^2 e^{-|Z|^2}  \d\bgamma(x)\d\bgamma(Z) 
\le e^{- 1/\tilde{\varepsilon}_N} \|f\|_{L^2}^2,
\end{multline*}
so that, since $\tilde{\varepsilon}_N\to 0$,
\[
\| [P_N,g] \|_{\rm H}^2  
= \frac{1}{\tilde{\varepsilon}_N^2} \iint \1\{x\in\mathcal S\} |f(x)-f(x+Z \tilde{\varepsilon}_N)|^2 e^{-|Z|^2}  \d\bgamma(x)\d\bgamma(Z) 
+ \underset{N\to\infty}{o(1)} .
\]
Now, since $f \in C^1$ and $\tilde{\varepsilon}_N \to0$ as
$N\to\infty$, one has  \[
  \frac{1}{\tilde{\varepsilon}_N^{2}}
  |f(x)-f(x+Z\tilde{\varepsilon}_N)|^2 \to |\nabla f(x)|^2 |Z|^2\qquad
  \text{locally uniformly for }(x,Z)\in \C^2.
\]
The integrand has bounded support in the $x$ variable and therefore, by
dominated convergence in the $Z$ variable,
\[
\| [P_N,g] \|_{\rm H}^2 \to \iint \1\{x\in\mathcal S\}  |\nabla f(x)|^2 |Z|^2  e^{-|Z|^2}  \d\bgamma(x)\d\bgamma(Z) 
= \int  |\nabla f|^2 \d\bgamma  . \qedhere
\]
\end{proof}

\begin{remark} \label{rk:ginvar}
For a linear statistic (for a suitable test function $f$) of the $\infty$-Ginibre ensemble, one has  
${\rm Var}_{\Pi_N}[\X(f)] = \tfrac12 \| [P_N,f] \|_{\rm H}^2$. 
By \eqref{Gvar}, the norm of this commutator has another probabilistic
interpretation as
\[
\| [P_N,f] \|_{\rm H}^2 = N \int  {\rm Var}_{\rm Z}[f(x+ \varepsilon_N{\rm Z})] \bgamma(\d x)
\]
where ${\rm Z}$ is a standard complex Gaussian. Thus, if $f_N$ is
bounded and supported in a compact set $\mathcal{A}$ independent of $N$,
one has
\[
{\rm Var}_{\Pi_N}[\X(f_N)]  \lesssim N \|f_N\|^2.
\]
If in addition $f_N$ is continuous with a modulus of continuity
$\omega_N$, then 
\[
{\rm Var}_{\Pi_N}[\X(f_N)]  \lesssim N \int_{\R_+}\omega_N(\varepsilon_N u )^2e^{-u} \d u
\]
Typical applications include Lipschitz or Hölder functions where the
estimate depend on $N$.
\end{remark} 

\subsection{Decay of the correlation kernel.}
\label{sec:decay_kernel}

The goal of this section is to prove that $\Pi_N$ is exponentially close to $P_N$
inside the bulk $\{V<\mu\}$, and that $\Pi_N$ is exponentially small in the \emph{forbidden region} $\{V>\mu\}$.
These estimates holds if $V$ satisfies  Assumptions~\ref{ass:V} a) (the conditions b) and c) are not relevant in this section).

Without loss of generality, by adding a constant to $V$  and
$\mu$ accordingly, we may assume that $\log(V)$ is
Lipschitz-continuous on $\C$. In particular, the potential $V$ is
positive. Our main results are the two following estimates.

\begin{proposition} \label{prop:decay_kernel} 
Let $V:\C\to \R$ be such that $\log(V)$ is Lipschitz-continuous and
let $\mu>\min(V)$. 
For every $\alpha>0$ and for every $\delta>0$ such that $\{V\le
\mu(1+\delta)\}$ is  compact, 
there exists $N_0(V,\mu,\delta,\alpha) \ge 1$ such that for any $N\geq
N_0$ and $(x,z)\in\C^2$, 
\begin{align}
\label{Pi_decay} |\Pi_N(x,z)|&\leq  N \exp(-\alpha\sqrt{N}[\mathrm{dist}(x,\{V\leq \mu+\delta\})+\mathrm{dist}(z,\{V\leq \mu+\delta\})]) \, , \\
\label{eq:Pi_close_P_bulk} |\Pi_N(x,z)-P_N(x,z)|&\leq  N \exp(-\alpha\sqrt{N}[\mathrm{dist}(z,\{V\geq \mu-\delta\})+\mathrm{dist}(x,\{V\geq \mu-\delta\})]) \, .
\end{align}
\end{proposition}

\begin{proof}
The bound \eqref{Pi_decay} is a consequence of Proposition
\ref{prop:Pi_outside+Weyl} below. Fix $\delta>0$ and let
\[
  \mathcal{D}_\delta^+ =\{z\in\C,\dist(z,\{V\leq (1+\delta)\mu\})\le
  \delta\}.
\]
If $\{V\le \mu(1+\delta)\}$ is compact, by Proposition
\ref{prop:Pi_outside+Weyl}, for every $\alpha>0$, for $z\notin
\mathcal{D}_{\delta}^+$ and $N$ large enough (depending on $\alpha,
\delta,\mu,V$), one has
\[
|\Pi_N(z,z)| \leq  C_0(V,\mu,\delta,\alpha)\exp(-\alpha\sqrt{N}\dist(z,\{V\leq \mu(1+\delta)\})).
\]
Let $L$ denote a Lipschitz constant of $V$ near $\mu$. With
$M=(L+1)(\mu+1)$ one has
\[
\mathcal{D}_\delta^+ \subset \{V\le \mu+M\delta\} .
\]
In particular, for every $z\in \C$, for $N$ large enough,
\[
|\Pi_N(z,z)| \le N  \exp(-\alpha\sqrt{N}\dist(z,\{V\leq \mu+M\delta\})) . 
\]
Here, we used the trivial bound \eqref{CSker} if $z\in\{V\leq
\mu+\ell\delta\}$, and chose $N\geq C_0$.
Replacing $\delta$ with $\delta/M$, by Cauchy-Schwarz \eqref{CSker},
we obtain \eqref{Pi_decay}.

\medskip

Similarly, the bound \eqref{eq:Pi_close_P_bulk} is a direct
consequence of Proposition~\ref{prop:Pi_P} with $\delta$ fixed. Letting
\[
 \mathcal{D}_\delta^-=\{x\in\C,\dist(x,\{V \ge \mu-\delta\})\ge
 \delta\},
\]
one has
\[
\{V\le \mu-M\delta\} \subset \mathcal{D}_\delta^-,
\]
and for every $\alpha>0$, for $N$ large enough, for every $z\in\mathcal{D}_\delta^-$, 
\[
0\le N-\Pi_N(z,z)
\leq C_0(V,\mu,\delta,\alpha)  {\rm exp}(- \alpha\delta \sqrt{N}).
\]

Replacing $\delta$ with $\delta/M$, using the trivial bound
\eqref{CSker2} on $\{V>\mu-M\delta\}$, and choosing $N\geq C_0$, we
obtain that for every $z\in\C$,
\[
0\le N-\Pi_N(z,z) \le   N \exp(-\alpha\sqrt{N}\mathrm{dist}(z,\{V\geq \mu-\delta\})) 
\]
By Cauchy-Schwarz, this yields \eqref{eq:Pi_close_P_bulk}.
\end{proof}
\begin{remark}Propositions \ref{prop:Pi_outside+Weyl} and \ref{prop:Pi_P}
  extend to the case where $\delta\to 0$ at a rate slower than
  $N^{-\frac 12}$, so Proposition \ref{prop:decay_kernel} is valid in
  this regime; the exponential decay rate becomes
  $-\alpha\delta\sqrt{N}\dist(\cdot,\{V\leq \mu+\delta\})$.
\end{remark}

The main ingredient to prove Proposition~\ref{prop:decay_kernel} is
the exponential concentration of the eigenfunctions of
Berezin--Toeplitz operators following the results established in
\cite{deleporte_fractional_2020}. Let us first prove a variant of
\cite[Proposition 4.1]{deleporte_fractional_2020} specialised to our setting.

\begin{proposition}\label{prop:decay_efs}
Let $N\geq 1$, let $\alpha\ge 0$, and let $\rho : \C \to \R$ be a $1$-Lipschitz function.
There is a constant $C_\alpha\ge 1$ so that for any eigenpair $(\lambda,u)$ of $H_N =P_N VP_N$, 
\[
\int_{\C}e^{\alpha\rho/\varepsilon_N}((1-C_\alpha\varepsilon_N)V-\lambda)|u|^2\leq 0
\qquad\text{and}\qquad
\int_{\C}e^{\alpha\rho/\varepsilon_N}(\lambda-(1+C_\alpha\varepsilon_N)V)|u|^2\leq 0.
\]
\end{proposition}

\begin{proof}
Define  $P_N^{\alpha} :=e^{\alpha\rho/\varepsilon_N}P_Ne^{-\alpha\rho/\varepsilon_N}$ as in Lemma~\ref{lem:Hop}. 
Using that $(\lambda,u)$ is an eigenpair of $P_NVP_N$, one has
\begin{align*}
\int e^{2\alpha\rho/\varepsilon_N}(V-\lambda)|u|^2
&= \langle e^{\alpha\rho/\varepsilon_N}u,e^{\alpha\rho/\varepsilon_N}(V-\lambda)P_Nu\rangle\\
&=\langle
e^{\alpha\rho/\varepsilon_N}u,e^{\alpha\rho/\varepsilon_N}(1-P_N)VP_Nu\rangle\\
&=\langle
e^{\alpha\rho/\varepsilon_N}u, (1-P_N^{\alpha})VP_N^{\alpha}e^{\alpha\rho/\varepsilon_N}u\rangle\\
&=\langle e^{\alpha\rho/\varepsilon_N}\sqrt{V}u,V^{-\frac
12} (1-P_N^{\alpha})VP_N^{\alpha}P_N^{\alpha}V^{-\frac 12}e^{\alpha\rho/\varepsilon_N}\sqrt{V}u\rangle .
\end{align*}
Our goal is now to show that the operator
$A  := V^{-\frac12}[V,P_N^{\alpha}]P_N^{\alpha}V^{-\frac 12} $
is appropriately bounded, that is $\|A\| \lesssim \varepsilon_N$ with a constant depending only on $\alpha$ and $\log(V)$. 
Using this bound, introducing a commutator, for a constant $C_\alpha\ge 1$,
\[
\bigg| \int e^{2\alpha\rho/\varepsilon_N}(V-\lambda)|u|^2 \bigg| \le C_\alpha \varepsilon_N \int e^{2\alpha\rho/\varepsilon_N} V |u|^2 
\]
which proves both claims.

Since we assume that $\log(V)$ is Lipschitz-continuous, there is a constant $K>0$ so that for $(x,y)\in\C^2$, 
\[
V(x)\leq V(y)\exp(K|x-y|)
\]
then
\[
|V(x)-V(y)|\leq V(x)|1-\exp(K|x-y|)|  \leq K V(x)|x-y|\exp(K|x-y|).
\]
By Lemma~\ref{lem:Hop}, this implies that
\begin{align*}
|[V,P_N^{\alpha}](x,y)|
&=|V(x)-V(y)|P_N^{\alpha}(x,y)\\
&\lesssim NV(x)|x-y|\exp(K|x-y|- N|x-y|^2/4).
\end{align*}
Then, with $A=V^{-\frac12}[V,P_N^{\alpha}]P_N^{\alpha}V^{-\frac 12}$, there is a small numerical constant $c>0$ so that
\[\begin{aligned} 
|A(x,z)| & \lesssim N \sqrt{V(x)/V(z)} \int  |x-y|\exp(K|x-y|-N|x-y|^2/4- N|y-z|^2/4) \d y  \\
&\lesssim \exp(- c N|x-z|^2)
\end{aligned}\]
for some numerical constant $c>0$ (the integrand concentrates for $y$ in a $\varepsilon_N$-neighborhood of $x$). 
Then, using Schur test, this yields $\|A\| \lesssim \varepsilon_N$ as claimed.
\end{proof}
From these weighted estimates, we can obtain some upper bounds on the
eigenvalue counting function of $H_N=P_NVP_N$, in the spirit of the
Weyl law, and deduce the decay of the spectral indicator function in
the forbidden region.  

\begin{proposition}\label{prop:Pi_outside+Weyl}Suppose $V:\C\to \R$
  satisfies Assumptions \ref{ass:V}. 
Let $\mu>0$ and $\delta>0$ such that $\{V\le \mu(1+\delta)\}$ is compact.
For $\sigma>0$ and $\eta>0$, let 
\[\mathcal{D}_\eta^+(\sigma) :=\{z\in\C,\dist(z,\{V\leq \sigma(1+\eta)\})\le \eta\}.\]
Let $(\delta_N)_{N\ge 1}$, non-increasing, such that $\delta_N \ge
\varepsilon_N \log N $. There exists $C>0$ such that
the spectral counting function
\[
  \mathcal{N}(\sigma):=\tr[\1(H_N\leq \sigma)]
\]
satisfies, for all $0<\sigma \le \mu+\delta$,
\begin{equation}\label{Weylbd}
\mathcal{N}(\sigma)\leq C N \bgamma(\mathcal{D}_{\delta_N}^+(\sigma)).
\end{equation}
Moreover, for any $\alpha >0$, there exists $C_{\alpha}$ and
$N_{\alpha}>0$ such that, for every $N\ge N_\alpha$ and $z\in
\C\setminus \mathcal{D}_{\delta_N}^+(\mu)$, 
\[
\Pi_N(z,z)  \le C_{\alpha} \exp(-\alpha\sqrt{N}\dist(z,\{V\leq \mu(1+\delta_N)\})) . 
\]
\end{proposition}

\begin{proof}
According to Lemma~\ref{lem:eig},  the spectrum of $H_N$ below
$\mu+\delta$ only consists of eigenvalues with finite multiplicity.

Let \[\rho:x\mapsto \dist(x,\{V\leq (1+\delta_N)\mu\}),\] and let us apply the first bound of Proposition~\ref{prop:decay_efs} to any normalised eigenfunction $u$ of $H_N$
with eigenvalue $\lambda\leq \mu$. Decomposing the integral in two
regions, for any $\alpha>0$, there exists a constant $C_{\alpha}$ such
that, for all such eigenfunctions,
\[
\int_{\{V>\mu(1+\delta_N)\}}\hspace{-.5cm}e^{\alpha\rho/\varepsilon_N}\underbrace{((1-C_\alpha\varepsilon_N)V-\lambda)}_{\geq \delta_N\mu/2\text{ if }N\ge N_0}|u|^2\leq \int_{\{V\leq
\mu(1+\delta_N)\}}\underbrace{(\lambda -(1-C_\alpha\varepsilon_N)V)}_{\leq \mu\text{ if }N\ge N_0}|u|^2.
\]
If $N$ is larger than some $N_{\alpha}$ (which silently also depends on the
family $(\delta_N)_{N\geq 1}$, and $V$), one has
$C_\alpha\varepsilon_N \le \min( \delta_N/4, 1)$. Thus, if $N\ge N_\alpha$,  
\begin{equation}\label{eq:bd_efs_outside}
\int_{\{V>\mu(1+\delta_N)\}}\hspace{-1cm}e^{\alpha \rho/\varepsilon_N}|u|^2\leq  2 \delta_N^{-1}.
\end{equation}
Let $\mathcal{A} = \mathcal{D}_{\delta_N}^+(\mu)$.
By construction,  $\rho\ge \delta_N$ on $\mathcal{A}$, so that summing
the inequality \eqref{eq:bd_efs_outside} over an orthogonal basis of
eigenfunctions 
in the range of $\1(P_NVP_N\leq \mu)$, since $\mathcal{A}^{\rm
  c}\subset\{V>\mu(1+\delta_N)\}$, we obtain with $\eta_N =\varepsilon_N\delta_N^{-1}$, 
\[\begin{aligned}
e^{\alpha/\eta_N} \int_{\mathcal{A}}\1(P_NVP_N\leq \mu) \d \bgamma
&\le\int_{\{V>\mu(1+\delta_N)\}}\hspace{-1cm}e^{\alpha\rho/\varepsilon_N}\, \1(P_NVP_N\leq \mu) \d \bgamma\\
&= \sum_{k\le\mathcal{N}(\mu)} \int_{\{V>\mu(1+\delta_N)\}}\hspace{-1cm}e^{\alpha \rho/\varepsilon_N}|{\rm u}_k|^2 \d\bgamma
\le  2\delta_N^{-1}\mathcal{N}(\mu).
\end{aligned}\]
Then, using the trivial upper-bound \eqref{kerptwbd}, 
\[\begin{aligned}
\mathcal{N}(\mu)=\int_{\C}\Pi_N(x,x) \d\bgamma(x)
\le  N \bgamma(\mathcal{A})+ 2\delta_N^{-1}e^{-\alpha/\eta_N}\mathcal{N}(\mu). 
\end{aligned}\]
If we choose $\alpha\geq \frac 12$, then
\[
  \delta_N^{-1}e^{-\alpha/\eta_N}\to 0
\]
and in particular, this quantity is bounded. We conclude that 
\[
\mathcal{N}(\mu)  \lesssim  N \bgamma(\mathcal{D}_{\delta_N}^+(\mu)) . 
\]
This proves \eqref{Weylbd}; the same argument applies to any energy $\sigma\le \mu+\delta$.

From here, we deduce \eqref{Pi_decay} by a summation argument. Using the pointwise estimate \eqref{uptwbd}, we have for $z\in\C$, 
\begin{equation} \label{kerdecayptw} 
|\Pi_N(z,z)|=  {\textstyle \sum_{k=1}^{\mathcal{N}(\mu)}} |{\rm u}_k(z)|^2  
\le  N {\textstyle \sum_{k=1}^{\mathcal{N}(\mu)}} \|{\rm u}_k\1_{\D(z,\varepsilon_N)}\|_{L^2}^2 .
\end{equation}
If $z$ is such that $\D(z,\varepsilon_N)\subset \{V>\mu(1+\delta_N)\}$,
applying \eqref{eq:bd_efs_outside}, we obtain, 
\[
\max_{k\le \mathcal{N}(\mu)}\|{\rm u}_k\1_{\D(z,\varepsilon_N)}\|_{L^2}^2 \lesssim\delta_N^{-1} e^{-\alpha \rho(z)/\varepsilon_N} 
\]
where we used the fact that, since $\rho$ is $1$-Lipschitz, 
$\displaystyle
\max_{\D(x,\varepsilon_N)} e^{\alpha(\rho(x)-\rho(z))/\varepsilon_N} \lesssim 1 . 
$

Hence, plugging the upper-bound $\mathcal{N}(\mu)\lesssim N$, we conclude that for $z\notin \mathcal{D}_{\delta_N}^+(\mu)$,
\[
|\Pi_N(z,z)|\lesssim N^2 \delta_N^{-1} e^{-\alpha\rho(z)/\varepsilon_N}. 
\]
Adjusting the constants, this proves the claim.
\end{proof}

Our hypotheses on the potential  (Assumption~\ref{ass:V}) are rather mild
at infinity, but for some refined properties used in the next sections
(notably the calculus of Berezin--Toeplitz operators), $V$ needs to be
smooth everywhere and/or growing at infinity. 
A consequence of Proposition \ref{prop:Pi_outside+Weyl} is a
\emph{replacement principle} for $V$; the point is that the spectral
theory below energy $\mu$ does not depend on $V$ on $\{V\geq
\mu+\delta\}$ for $\delta>0$, up to an exponentially small error. 

\begin{lemma}[Replacement principle]\label{prop:replacement}Let $V$,
  $\delta$, $\mu$
  satisfy Assumptions \ref{ass:V}. 
Let $V_1:\C\to \R_+$ be such that and $\log(V_1)$ is Lipschitz
continuous and
compact and $V=V_1$ on $\{V\leq \mu(1+\delta)\}$.

Let $(\delta_N)_{N\ge 1}$, non-increasing, with $ N^{-\kappa} \le
\delta_N  \le \delta$ for some $\kappa< \frac 12$.
Then, for any $\alpha>0$, there exists $C_{\alpha}>0$ and
$N_{\alpha}\in \N$
such that for every $N\geq N_{\alpha}$,
\begin{equation*}
\|\1(P_NVP_N\leq \mu)\1(P_NV_1P_N\geq \mu+\delta_N)\|_{\tr} \le C_\alpha  e^{-\alpha\delta_N\sqrt{N}} .
\end{equation*}
\end{lemma}

\begin{proof}
Let  $\rho=\dist(\cdot,\{V\leq \mu(1+\delta_N/2)\})$ and let $V_2:= V_1-V$.
Since $V_2$ vanishes on $\{V>\mu(1+\delta_N)\}$, if $({\rm u},\lambda)$ is a normalised eigenfunction of $P_NVP_N$ with eigenvalue $\lambda \leq \mu(1+\delta_N/2)$,  one can bound for $z\in\C$, 
\[
\bigg| \int P_N(x,z) V_2(x) {\rm u}(x)  \bgamma(\d x) \bigg|^2
\le \underbrace{ \int_{\{V>\mu(1+\delta)\}}\hspace{-.5cm} e^{\alpha\rho/\epsilon_N}  |{\rm u}|^2}_{\lesssim 1 \text{ by \eqref{eq:bd_efs_outside} with $\delta$ fixed}}
\int  |P_N(x,z)|^2  e^{-\alpha\rho(x)/\varepsilon_N} V_2^2(x) \bgamma(\d x) .
\]
Then, since ${\rm u}$ is an eigenfunction and using the reproducing property of~$P_N$, one has
\[
\|P_NV_2P_N{\rm u}\|_{L^2}^2 = \int \bigg| \int P_N(x,z) V_2(x) {\rm u}(x)  \bgamma(\d x) \bigg|^2 \bgamma(\d z)
\lesssim N  \int    e^{-\alpha\rho(x)/\varepsilon_N} V_2^2(x)  \bgamma(\d x).
\]
There exists $c>0$, depending only on $\mu$ and the Lipschitz constant
of $V$, such that
\[
  \dist(\{V\geq \mu(1+\delta_N)\},\{V\leq \mu(1+\tfrac {\delta_N}
  2)\})>3c\delta_N;
\]
therefore, since $V_2$ vanishes on $\{V\leq \mu(1+\delta_N)\}$, and
since $\log(V_2)$ is Lipschitz with Lispchitz constant $\Lambda$, if
$N$ is larger than some $N_0$ depending on $\alpha$ and
$(\delta_N)_{N\geq 1}$, one can control
\begin{align*}
  N\int e^{-\alpha\rho(x)/\varepsilon_N} V_2^2(x)  \bgamma(\d x)
  &\leq
  N\mu(1+\delta)\int_{\{\rho(x)>2c\delta_N\}}e^{-\alpha\rho(x)/\varepsilon_N}e^{2\Lambda
    \rho(x)}\dd \gamma(x)\\
  &\leq
    N\mu(1+\delta)\int_{3c\delta_N}^{+\infty}e^{(-\alpha/\varepsilon_N+2\Lambda)t}(C_0(\mu)+t)\dd
    t\\
  &\leq NC_1(\mu)e^{-2c\alpha\delta_N/\varepsilon_N}\\
  &\leq e^{-c\alpha\delta_N/\varepsilon_N}.
\end{align*}
Thus,
\begin{equation}\label{eq:replace_eigenfunctions}
\|P_N(V_1-\lambda)P_N{\rm u}\|_{L^2}=\|P_NV_2P_N{\rm u}\|_{L^2}  \leq e^{-c\alpha\delta_N/\varepsilon_N}.
\end{equation}
In particular, ${\rm u}$ almost solves the eigenvalue equation for $P_NV_1P_N$. 
Since this operator is self-adjoint, by spectral stability, we conclude that for any eigenfunction $({\rm u},\lambda)$ of $P_NVP_N$ with $\lambda\leq  \mu(1+\delta/2)$, 
\[
\|\1\{P_NV_1P_N> \mu+\delta_N\}{\rm u}\|_{L^2}^2  \leq  \frac{2}{\delta_N}e^{-c\alpha\delta_N/\varepsilon_N}  . 
\]
By positivity,
\[\begin{aligned}
\| \1(P_NVP_N\leq \mu)\1(P_NV_1P_N\geq \mu+\delta_N) \|_{\tr} &= \sum_{k\le \mathcal{N}(\mu)}
\big\langle {\rm u}_k | \1\{P_NV_1P_N\geq \mu+\delta_N\} {\rm u}_k\big\rangle \\
& \le \sum_{k\le \mathcal{N}(\mu)} \|\1\{P_NV_1P_N\geq \mu+\delta_N\}{\rm u}_k\|_{L^2}^2
\leq \frac{2\mathcal{N}(\mu)}{\delta_N} e^{-c\alpha\delta_N/\varepsilon_N}  . 
\end{aligned}\]
Since  $\{V\le \mu(1+\delta)\}$ is compact, by \eqref{Weylbd}, one has
$\mathcal{N}(\mu)=\O(N)$. Adjusting constants
($\alpha>0$ is arbitrary), this proves the claim. 
\end{proof}

\begin{remark} \label{rk:specstab}The main application of Proposition
\ref{prop:replacement} will be the following.
Suppose that $V$ satisfies Assumptions \ref{ass:V} and let $V_1\in
C^{\infty}(\C,\R)$ such that $V_1= V$ on $\{V\leq \mu(1+\delta)\}$ and
$V_1=\mu(1+2\delta)$ on $\{V\ge \mu(1+2\delta)\}$. If $({\rm u},
\lambda)$ is an eigenpair of $P_NVP_N$ with energy
$\lambda\le \mu+\delta/2$, by  \eqref{eq:replace_eigenfunctions},  
for any normalised eigenfunction $({\rm u}^1_k,\lambda^1_k)$ of
$P_NV_1P_N$, one has
\[
(\lambda_k^1-\lambda)  \langle {\rm u}^1_k, {\rm u} \rangle =  \langle
P_N(V_1-V)P_N {\rm u}^1_k, {\rm u} \rangle \leq C_{\alpha} e^{-2\alpha\sqrt{N}}  .
\]
In particular, either eigenvalues are close to each other:
$|\lambda^1_k-\lambda| \lesssim  e^{-\alpha\sqrt{N}}$ or the scalar product
is small: $|\langle {\rm u}^1_k, {\rm u} \rangle | \lesssim e^{-\alpha\sqrt{N}}$.

The point is that the eigenvalues of $\Pi_NV_1\Pi_N$ in the range
$[\mu(1-\delta),\mu(1+\delta)]$ are simple and separated by about
$N^{-1}$. This is a consequence of the results of
\cite{charles_quasimodes_2003}; one can also see it by conjugation to
a pseudodifferential operator on $\R^2$ with symbol $V_1$ and
semiclassical parameter $N^{-1}$ (\cite{zworski_semiclassical_2012},
Theorem 13.10).

Consequently, given such $({\rm u}, \lambda)$ with $\lambda\in
[\mu(1-\delta),\mu(1+\delta)]$, there exists exactly one index $k$ such
that $|\lambda^1_k-\lambda| \lesssim  e^{-\alpha\sqrt{N}}$, and one
has then  $\|{\rm u}^1_k- {\rm u} \|_{L^2} \lesssim N^2
e^{-\alpha\sqrt{N}} $ for some eigenfunction ${\rm u}_k^1$ with
eigenvalue $\lambda_k$.

In particular, the eigenvalues of  $P_NVP_N$ in
$[0,\mu(1+\delta/2)]$ are also separated by $N^{-1}$. Then, reversing
the argument, we deduce that the eigenvalues of both operators lie
exponentially close to each other,
with exponentially close eigenfunctions. In other words,  
there is a one-to-one mapping between the spectral data of $\Pi_NV\Pi_N$ and $\Pi_NV_1\Pi_N$, in the range $[\mu(1-\delta),\mu(1+\delta)]$, which is exponentially close to identity.
\end{remark}

Now we turn to the approximation of $\Pi_N$ inside the bulk (for this proof, we need the \emph{replacement principle}). 

\begin{proposition}\label{prop:Pi_P}
Let $\mu>0$, $\delta>0$ and $V$ satisfying  Assumption~\ref{ass:V}.
For $\eta>0$, let 
\begin{equation} \label{D_bulk}
\mathcal{D}_\eta^-(\mu) :=\{x\in\C,\dist(x,\{V \ge \mu-\eta\})\ge \eta\} .
\end{equation}
Let $(\delta_N)_{N\ge 1}$, non-increasing, such that $\delta_N \ge \varepsilon_N \log N $. 
Then, for any $\alpha>0$, there exists $C_{\alpha}$ and $N_{\alpha}$
such that, for all $N\geq N_{\alpha}$,
\[
\sup_{z\in\C}\sup_{x\in \mathcal{D}_{\delta_N}^-(\mu)}|P_N(x,z) -\Pi_N(x,z)| \le C_\alpha {\rm exp}(- \alpha\delta_N\sqrt{N}).
\]
\end{proposition}

\begin{proof} We first assume that $\log(V)$ is Lipschitz and $V$
  grows at infinity, so that for any $\sigma>0$, the set
  $\{V\leq\sigma\}$ is compact (or empty).  We will remove this
  condition at the end of the proof using Proposition \ref{prop:replacement}.  
In this case, by Lemma~\ref{lem:eig}, the whole spectrum of $H_N =P_NVP_N$ is discrete (by Proposition \ref{prop:Pi_outside+Weyl},  all eigenvalues have finite multiplicity and $\lambda_n \to\infty$ as $n\to\infty$).

Let $({\rm u}_n, \lambda_n)_{n\in\N_0}$ a spectral decomposition of $H_N=P_NVP_N$ and let 
\[
\rho(x)  = \pm \dist(x,\{V=\mu(1-\delta_N)\})) , \qquad \pm =
\operatorname{sgn}(\mu(1-\delta_N)-V(x)). 
\]
This function is $1$-Lipschitz, so  applying the second inequality of Proposition \ref{prop:decay_efs}, if $\lambda_n\ge
\mu$, 
\[
\int_{\{V\le \mu-\delta_N\}} \hspace{-.7cm} e^{\alpha\rho/\varepsilon_N} \underbrace{(\lambda-(1+C_\alpha\varepsilon_N)V)}_{\ge \delta_N/2 \text{ if }N\ge N_\alpha} |{\rm u}_n|^2 \le \underbrace{(1+C_\alpha\varepsilon_N)}_{\le 2\text{ if }N\ge N_\alpha} \int_{\{V >  \mu-\delta_N\}}  \hspace{-.7cm}e^{\alpha\rho/\varepsilon_N} V|{\rm u}_n|^2 . 
\]
using that by assumption, $C_\alpha\varepsilon_N \le \delta_N/2 \wedge 1$ if  $N\ge N_\alpha$ (for some fixed $N_\alpha$).
Since $V$ is log-Lipschitz on $\C$, $\rho$ is negative in the region $\{V >  \mu-\delta_N\}$, for some constant depending only on $\alpha>0$, 
\[
e^{\alpha\rho/\varepsilon_N} V \lesssim  e^{\alpha\rho/2\varepsilon_N} . 
\]

Now, for $x\in \mathcal{D}_\delta^-$, one has $\D(x,\varepsilon_N) \subset \{V\le \mu-\delta)\}$  (for any $\delta>\varepsilon_N$) and 
$\displaystyle
\max_{\D(x,\varepsilon_N)} e^{\alpha(\rho(x)-\rho(z))/\varepsilon_N} \lesssim 1 
$, so that
\[
\|{\rm u}_n\1_{\D(x,\varepsilon_N)}\|_{L^2}^2 \lesssim e^{-\alpha\rho(x)/\varepsilon_N} \int_{\{V\le \mu-\delta_N\}} \hspace{-1cm} e^{\alpha\rho/\varepsilon_N}  |{\rm u}_n|^2 
\lesssim  e^{-\alpha\rho(x)/\varepsilon_N}  \int_{\{V >  \mu -\delta_N \}}  \hspace{-1cm} e^{\alpha\rho/2\varepsilon_N}  |{\rm u}_n|^2 . 
\]
Then, summing the previous estimate over a spectral decomposition of $P_NVP_N$ and using \eqref{uptwbd}, we obtain  for $x\in \mathcal{D}_{\delta_N}^-$, 
\[\begin{aligned}
\1(H_N> \mu)(x,x)
=\sum_{\lambda_n\geq \mu}| {\rm u}_n(x)|^2 
&\le  N\sum_{\lambda_n\geq \mu}\|{\rm u}_n\1_{\D(x,\varepsilon_N)}\|_{L^2}^2\\
&\lesssim  N e^{-\alpha\rho(x)/\varepsilon_N}   \int_{\{V >  \mu -\delta_N\}}    \hspace{-1cm} e^{\alpha\rho/2\varepsilon_N} 
\sum_{n\ge 0}|{\rm u}_n|^2 \d\bgamma .
\end{aligned}\]
Using the trivial bound \eqref{kerptwbd}, the above integral is $\O(N)$. We conclude that for $x\in \mathcal{D}_{\delta_N}^-$,
\[
\1(H_N> \mu)(x,x)\lesssim N^2  e^{-\alpha\rho(x)/\varepsilon_N} . 
\]
Consequently, if  $\delta_N \ge \varepsilon_N \log N $, integrating this estimate, 
\begin{equation} \label{trextdecay}
\|\1(H_N>\mu)\1\{\mathcal{D}_{\delta_N}^-\}\|_{\tr} \lesssim e^{-\alpha\delta_N/2\varepsilon_N} . 
\end{equation}
Note that this bound holds replacing $(\mu, \mathcal{D}_{\delta}^-) \leftarrow  (\mu-\delta ,\mathcal{A}_{\delta})$ with
$\mathcal{A}_\delta = \{x\in\C,\dist(x,\{V \ge \mu-2\delta\})\ge \delta\}$ and $\delta=\delta_N$.

To remove the growth assumption on $V$; let  $\hat{V}$ satisfying Assumption~\ref{ass:V} a), which is
equal to $V$ on $\{V\geq \mu\}$, and define $\hat{H}_N = P_N\hat{V}P_N$. 
Then, by Lemma~\ref{prop:replacement} and \eqref{trextdecay}, for
every $\alpha\geq 0$, there exists $C_{\alpha}$ such that, for $N$
large enough,
\[\begin{aligned}
\|\1(\hat{H}_N>\mu)\1\{\mathcal{A}_\delta\}\|_{\tr} &\le  \| \1(\hat{H}_N>\mu)\1(H_N\le\mu-\delta)\1\{\mathcal{A}_\delta\} \|_{\tr}
+\|\1(\hat{H}_N>\mu)\1(H_N>\mu-\delta)\1\{\mathcal{A}_\delta\} \|_{\tr} \\
&\le \|\1(\hat{H}_N>\mu)\1(H_N\le\mu-\delta)\|_{\tr} +
\|\1(H_N>\mu-\delta)\1\{\mathcal{A}_\delta\} \|_{\tr}  \leq C_{\alpha} e^{-\alpha\delta_N/2\varepsilon_N} .
\end{aligned}\]
Then, by \eqref{uptwbd} again, we have for $x\in \mathcal{D}_{2\delta_N}^-$,
\begin{equation} \label{extptwdecay}
\1(\hat{H}_N>\mu)(x,x) \le N  \| \1(\hat{H}_N>\mu) \1\{\D(x,  \varepsilon_N)\|_{\tr} \lesssim N e^{-\alpha\delta_N/2\varepsilon_N} 
\end{equation}
as $\D(x,  \varepsilon_N) \subset \mathcal{A}_\delta$ for $x\in \mathcal{D}_{2\delta_N}^-$ and using positivity.
Since $\1(H_N>\mu)=P_N -\Pi_N$, replacing $2\delta_N$ with $\delta_N$ and adjusting the constants, this completes the proof.
\end{proof}

To conclude this section, we use the estimates from  Proposition~\ref{prop:decay_kernel} to obtain the following result. 

\begin{lemma} \label{lem:pertPi}
Assume $V,\mu,\delta$ satisfy ~\ref{ass:V} a). Let $\rho :
\C \to[-1,1] $ be a 1-Lipschitz function supported in
$\{V\le\mu-\delta\}$. There exists $C_0(\delta,V,\mu)$ and $N_0(\delta,V,\mu)$
  such that,
for any $\alpha\in [0,1]$, for every $N\geq N_0$,
\[
\|e^{\alpha\rho/\varepsilon_N}\Pi_Ne^{-\alpha\rho/\varepsilon_N}-\Pi_N\|
\leq C_0\alpha . 
\]
\end{lemma}

\begin{proof}
We denote $g := \alpha\rho/\varepsilon_N$ and let $\chi =
\1\{V\le\mu+2\delta\}$. Since $e^{g}(1-\chi) =(1-\chi)$, we have
\[
e^{g}\Pi_Ne^{-g}-\Pi_N = (e^{g}\Pi_Ne^{-g}-\Pi_N)\chi  + (e^{g}-1)\Pi_N (1-\chi).
\]
Note also that $\sup|e^g-1|\leq \alpha\sqrt{N}e^{\alpha\sqrt{N}}$.

Let us first estimate $\|(e^g-1)\Pi_N(1-\chi)\|$. 
Let $\eta=\dist(\{V>\mu+\delta\},\{V<\mu-\delta\})$.
By \eqref{Pi_decay}, with $r : x \mapsto \mathrm{dist}(x,\{V\leq
\mu+\delta\})$, for all $\beta\ge 1$, there exists $c>0$,
$C_{\beta}>0$, and $N_{\beta}$ such that, for all $z\in \C$ and all
$N\geq N_{\beta}$,
\[\begin{aligned}
&(e^{g(z)}-1) \int (1-\chi)(x)|\Pi_N(x,z)|  \d\gamma(x) \leq \alpha C_{\beta}\sqrt{N}e^{\alpha\sqrt{N}} \int_{\{V\le\mu+2\delta\}} \hspace{-1cm}e^{- \beta r(x)\sqrt{N}}  \d\gamma(x)
\leq \alpha C_{\beta} \sqrt{N}e^{(\alpha-\beta\eta)\sqrt{N}} \, , \\
&(1-\chi)(z)\int (e^{g(x)}-1) |\Pi_N(x,z)|  \d\gamma(x) \leq C_{\beta}  e^{-c\eta\sqrt{N}}   \int (e^{g(x)}-1)  \d\gamma(x) 
\leq \alpha C_{\beta}\sqrt{N}e^{(\alpha-\beta\eta)\sqrt{N}} \, .
\end{aligned}\]
Now we fix $\beta>\frac 1\eta$, so that $\alpha<\beta\eta$.
Thus, for some $c'>0$ and $C>0$, for all $N$ large enough,
\[ 
\| (e^{g}-1)\Pi_N (1-\chi)\|   \leq C\alpha \sqrt{N}e^{-c' \sqrt{N}}. 
\]
Choosing $N$ and $C$ larger ensures
\[
\| (e^{g}-1)\Pi_N (1-\chi)\|   \leq C\alpha. 
\]
All in all, in operator norm,
\begin{equation} \label{bulkPiloc}
e^{g}\Pi_Ne^{-g}-\Pi_N = \chi(e^{g}\Pi_Ne^{-g}-\Pi_N)\chi+ \underset{N\to\infty}{\O}\big(\alpha\big) .
\end{equation}

Now, let us decompose
\[
e^{g}\Pi_Ne^{-g}-\Pi_N  = (e^{g}-1)\Pi_Ne^{-g}- \Pi_N (1-e^{-g}) .
\]
We have a similar decomposition replacing $\Pi_N$ by $P_N$, so that 
\begin{equation} \label{bulkextsplit}
\chi (e^{g}\Pi_Ne^{-g}-\Pi_N)\chi =  \chi (e^{g}P_Ne^{-g}-P_N)\chi + \chi (P_N -\Pi_N) (1-e^{-g})-  (e^{g}-1) (P_N -\Pi_N) \chi 
\end{equation}
as $\chi=1$ on $\supp(g)$ (inside the bulk). 
Since $g$ is supported in $\{V\le\mu-\delta\}$ for some $\delta>0$  and $\| e^g\| \le e^{\alpha\sqrt{N}}$,
by \eqref{eq:Pi_close_P_bulk}, letting
$\eta=\dist(\{V<\mu-\delta\},\{V>\mu-\frac \delta 2\})$, for any $\beta\ge 1$, there exists
$C_\beta$ such that, for $N\geq N_{\beta}$,
\[
| (e^{g(x)}-1) (P_N -\Pi_N)(x,z) \chi (z)|  \leq \alpha
C_{\beta}\sqrt{N}e^{(\alpha-\beta\eta)\sqrt{N}} \chi(z)\chi(x) , \qquad
(x,z)\in\C^2.
\]
Here we again used the fact that $\sup|e^g-1|\leq
\alpha\sqrt{N}e^{\alpha\sqrt{N}}$.

By the Schur test, choosing $\beta$ large enough and then $N$ large
enough, we have, for some $C>0$ and $c'>0$,
\[
\| (e^{g}-1) (P_N -\Pi_N) \chi \| + \| \chi (P_N -\Pi_N)
(1-e^{-g})\| \leq C \alpha  \sqrt{N}e^{-c'\sqrt{N}}. 
\]
Again, up to increasing $C$ and $N$, we obtain
\[
\| (e^{g}-1) (P_N -\Pi_N) \chi \| + \| \chi (P_N -\Pi_N)
(1-e^{-g})\| \leq C \alpha.
\]

Hence, by \eqref{bulkPiloc} and \eqref{bulkextsplit}, we conclude
that, in operator norm
\[
e^{g}\Pi_Ne^{-g}-\Pi_N =  \chi (e^{g}P_Ne^{-g}-P_N)\chi  + \underset{N\to\infty}{\O}\big(\alpha\big).
\]
By Lemma~\ref{lem:Hop}, this completes the proof.
\end{proof}

\subsection{Weyl law \& probabilistic consequences.}
We use Proposition \ref{prop:decay_kernel} to derive the pointwise Weyl law for $\Pi_N$. Then, we discuss some probabilistic consequences of these estimates, including Proposition~\ref{prop:LLN} and local universality. 

\begin{proposition}[Weyl law]\label{prop:Weyl}
Under Assumptions \ref{ass:V}, for $\bgamma$-almost every $x\in\C$,
\[ 
N^{-1}\Pi_N(x,x)\to \1\{V(x)<\mu\} , \qquad\text{ as $N\to +\infty$}.
\]
In particular, for any $g:\C\to\R$ continuous and bounded,
\[
\E_{\Pi_N}[\X(g)] = \int g(x) \Pi_N(x,x) \bgamma(\d x) \sim N \int_{\{V<\mu\}} \hspace{-.5cm}g \d\bgamma .
\]
\end{proposition}

\begin{proof}
By Proposition \ref{prop:decay_kernel}, using that $P_N(x,x)=N$, for
every $\delta>0$, there exists $c_\delta>0$ such that, as $N\to +\infty$, 
\[
N^{-1}\Pi_N(x,x) = \1\{V(x)<x\} +O(e^{-c_\delta\sqrt{N}}) \qquad\text{uniformly for } x\in\{|V-\mu|>\delta\}.
\]
Since the set $\{V=\mu\}$ is negligible under Assumptions~\ref{ass:V}
c), this implies the pointwise asymptotics and the weak convergence. 
\end{proof}

Another immediate consequence of the approximation of the correlation kernel is the universality of microscopic linear statistics. Namely, for any Berezin-Toeplitz ensemble, the local point process in the bulk converges (weakly) to the $\infty$-Ginibre ensemble with density 1. 

\begin{proposition}[Local universality]\label{prop:uni}
Recall that $\varepsilon_N=N^{-\frac 12}$ denotes the microscopic scale and let $(\delta_N)_{N\ge 1}$, non-increasing, satisfy $\delta_N \ge C \varepsilon_N \log N $ for a large constant $C\ge 1$. 
Let $f \in C^0_c(\C,\R)$, and let $g_N:=f\big((\cdot-x_N)\varepsilon_N^{-1}\big)$ where  $x_N \in \mathcal{D}_{\delta_N}^- $, \eqref{D_bulk}
Under Assumptions \ref{ass:V} a), in distribution as $N\to\infty$, 
\[
\{\X(g_N)\}_{\P_{\Pi_N}} \to \{\X(f)\}_{\P_{P_1}} \, .
\]
\end{proposition}

\begin{proof}
In distribution, $\{\X(g_N)\}_{\P_{\Pi_N}} = \{\X(f)\}_{\P_{\widehat{\Pi}_N}}$ where
\[
\widehat{\Pi}_N :(w,z) \in \C^2 \mapsto \Pi_N(x_N + w \varepsilon_N, x_N+ z\varepsilon_N )  \varepsilon_N^2 . 
\]
Note that under the same scaling, $\widehat P_N = P_1$, see \eqref{oP}. 
Then, by Proposition~\ref{prop:Pi_P}, $\widehat\Pi_N \to P_1$ locally
uniformly as kernels, and thus the Fredholm determinants
\eqref{def:Ups} converge by
Grümm's theorem:
for any test function $f$ bounded with compact support, 
\[
\Upsilon(g_N, \Pi_N) = \Upsilon(f, \hat\Pi_N) \to \Upsilon(f,P_1) \qquad\text{as $N\to\infty$}. 
\]
This is the definition of weak-convergence for point processes; $\P_{\hat\Pi_N} \to \P_{P_1}$. This also proves the claim. 
\end{proof}

We are now in position to prove the law of large numbers for
Berezin--Toeplitz ensembles under Assumptions \ref{ass:V}. 

\begin{proof}[Proof of Proposition~\ref{prop:LLN}.]~\\
\underline{Exponential decay.}
Let $\delta_N>C\varepsilon_N\log(N)$ for $C$ large enough and let
\[\eta_N=\dist(\{V\geq \mu(1+2\delta)\},\{V\leq
  \mu(1+\delta)\}>c_0\delta_N.\]
Observe that for  any function $f:\C \to\R_+$
supported in $\{V\leq \mu(1+2\delta_N)\}$ with at most exponential growth at infinity  (namely,
$f \le e^{C|\cdot|}$), by Proposition~\ref{prop:Pi_outside+Weyl}, for
every $\alpha>0$, there exists $C_{\alpha}$ and $N_{\alpha}$ such that, for every
$N\geq N_{\alpha}$,
\[
\E_N[\X(f)] = \int f(x) \Pi_N(x,x) \leq C_{\alpha}\int f(x)
e^{-\alpha\sqrt{N}\dist(x,\{V<\mu(1+\delta_N)\})}\leq 2C_{f,\alpha}e^{-\alpha\sqrt{N}\eta_N}.
\]
In particular,
\[
\P_N\big[ \exists \text{ a particle in } \{V\geq
\mu(1+2\delta_N)\}\big] \lesssim e^{-\alpha c_0\delta_N \sqrt{N}} ,
\]
that is, all particles lie in a $\O(\delta_N)$-neighborhood of the droplet $\{V\le\mu\}$  with overwhelming probability.

\smallskip
\noindent
\underline{Boundary layer.} By \eqref{kerptwbd}, the intensity of
$\mathbf{X}$ is dominated by $N$, so that
\[
  \mathbb{E}_N\big[\X(\{|V-\mu|\leq \delta_N\})\big]\leq
  N\bgamma(\{|V-\mu|\leq \delta_N\})\leq C N\delta_N.
\]

All in all, letting $\chi_N$ be a smooth cut-off satisfying
\[
  \1\{V\le \mu-\delta_N\} \le \chi\le \1\{V\le \mu-\delta_N/2\},
\]
for every $f:\C \to\R_+$  with at most exponential growth, one has
\begin{equation} \label{bdlayerl2}
N^{-1}\X(f) =  N^{-1}\X(f\chi) + \O_{N\to +\infty}(\delta_N )
\end{equation}
where the error is controlled in $L^q(\P_N)$ for any $q<\infty$.

\smallskip
\noindent
\underline{Variance estimates.}
For any determinantal process associated with a finite rank projection
$\Pi$, the following basic variance estimate holds: for any bounded
test function $f$,
\[
{\rm Var}_{\Pi}[\X(f)] \le \|f\|^2 \tr(\Pi) .
\]
With a localization on a neighborhood of the bulk, this implies that
for any function $f:\C \to\R_+$  with at most exponential growth, one has
\[
{\rm Var}_{\Pi_N}[\X(f)] \leq C_f N.
\]
This is a crude bound (in the bulk, this should be compared to the
estimate from Remark~\ref{rk:ginvar} for the $\infty$-Ginibre
ensemble, and ultimately with Theorem \ref{thm:CLT}) but it does not
require any regularity properties on $f$.

Together with Proposition~\ref{prop:Weyl}, this proves
Proposition~\ref{prop:LLN}.
\end{proof}

\subsection{Decorrelation estimates.}

In this section, we use the estimate the from Section~\ref{sec:decay_kernel} -- these estimates hold under Assumption~\ref{ass:V} a) on $V$ -- to prove some \emph{decorrelation estimates} for the  log-Laplace transform 
\[
\Upsilon(f;\Pi_N) = \log\det(1+\Pi_N(e^{f}-1)\Pi_N) -  \tr(\Pi_Nf\Pi_N)  .
\]
For instance, if $f$ is a linearly growing  function, supported in the
forbidden region $\{V\ge\mu+\delta\}$ for some $\delta>0$, then
$\Upsilon(f;\Pi_N) = \O(e^{-\alpha\sqrt{N}})$ for every $\alpha>0$; see
Lemma~\ref{lem:deco1}.  
Similarly, if $f_1,f_2$ are two bounded functions with disjoint
support (separated by $\delta>0$), then \[\Upsilon(f_1+f_2;\Pi_N) =
\Upsilon(f_1;\Pi_N)+\Upsilon(f_2;\Pi_N)+\O(e^{-\alpha
  \sqrt{N}});\] see Proposition~\ref{lem:deco2}.  
In particular, these asymptotics imply that the fluctuations of the
linear statistics $\X(f_1)$  and $\X(f_2)$ are independent up to an
exponentially small error.  

\medskip

We begin by giving a \emph{localisation} estimate in a neighborhood of the bulk. 

\begin{lemma}\label{lem:deco1}
Let $f, g: \C \to \R$. Suppose that $f$ is bounded, $g$ is supported in $\{V>\mu+\delta\} $ for some $\delta>0$ and $|g(\cdot)|\leq C(|\cdot|+1)$ for some constant $C\ge 1$. 
Then, for any $\alpha>0$, 
\[ 
\Upsilon(f+g;\Pi_N) = \Upsilon(f;\Pi_N) + \O(e^{-\alpha \sqrt{N}})  
\]
where the implied constant depends on $e^{\|f\|}$, $C$, $\alpha$ and $\delta$. 
\end{lemma}
\begin{proof}
Define the following operators:
\[
B=1-\Pi_N + \Pi_N e^fe^g \Pi_N \, , \qquad\quad
R=\Pi_N e^f[\Pi_N,e^g]\Pi_N \, ,
\]
\[
\quad B_1=1-\Pi_N +\Pi_N e^f\Pi_N \, , \qquad\quad \,  \, \,
B_2=1-\Pi_N +\Pi_Ne^g\Pi_N \, . 
\]
Observe that $B_1\ge 1-\Pi_N + e^{\min f}\Pi_N \ge e^{-\|f\|}$ as an operator ($f$ is uniformly bounded). Therefore, $B_1^{-1}$ is well-defined, bounded, with $\|B_1^{-1}\|\le e^{\|f\|}$.

Then $B_1B_2=B+R$ so that $\det(B) = \det(B_1B_2-R) =  \det(B_1)\det(B_2-B_1^{-1}R)$ and 
\[
\Upsilon(f+g,\Pi_N) =\Upsilon(f,\Pi_N)+\log \det(B_2-B_1^{-1}R)-\tr(g\Pi_N).
\]
Using the estimate \eqref{Pi_decay}, with   $\rho(x) =\alpha
\mathrm{dist}(x,\{V\leq \mu+\delta/2\})$ for all $\alpha>0$, there
exists $C_{\alpha}$ and $N_{\alpha}$ such that, for every $N\geq N_{\alpha}$,
\begin{equation}\label{trextdecay}
\begin{aligned}
\|(e^{g}-1)\Pi_N \|_{\tr} , \|\Pi_N(e^{g}-1)\|_{\tr} & \le \iint |e^g(x)-1| |\Pi_N(x,z)|   \bgamma(\d x) \bgamma(\d z) \\
&\le  \iint_{\{V(x)\ge \mu+\delta\}} \hspace{-.95cm}e^{|g(x)|}  |\Pi_N(x,z)|   \bgamma(\d x) \bgamma(\d z) \\
&\leq  N C_{\alpha}\iint_{\{V(x)\ge \mu+\delta\}} \hspace{-1cm}    e^{C|x| - \alpha\sqrt{N}(\rho(x)+\rho(z))}  \bgamma(\d x) \bgamma(\d z) \\
&\leq NC_{\alpha}e^{-\alpha\eta \sqrt{N}}
\end{aligned}
\end{equation}
where $\eta=\dist(\{V\geq \mu+\delta\},\{V\leq \mu+\delta/2\})$.

In particular, $\|[\Pi_N,e^g]\|_{\tr} \leq C_{\alpha}N
e^{-\alpha\eta\sqrt{N}} $ and similarly $\|g\Pi_N \|_{\tr} \leq C_{\alpha}N  e^{-\alpha\eta\sqrt{N}}  $.
Thus, using the bound $|\det(1+A)| \le \exp \|A\|_{\tr}$ for $A$ trace-class,
\[\begin{aligned}
|\log \det(B_2-B_1^{-1}R)| & \leq \| \Pi_N(e^g-1)\Pi_N - B_1^{-1}R \|_{\tr} \\
&\le \|(e^{g}-1)\Pi_N \|_{\tr} + \|B_1^{-1} \Pi_Ne^{f}\| \|[\Pi_N,e^g]\|_{\tr} \\
&\leq C_{\alpha} (1+  e^{2\|f\|}) e^{-\alpha\eta\sqrt{N}}  . 
\end{aligned}\]
Since $\eta$ is fixed, we can replace $\alpha$ with $\alpha/\eta$ and
the proof is complete.
\end{proof}

Now we give a \emph{decorrelation} estimate inside the bulk. 
In particular, we will use Proposition~\ref{lem:deco2} to separate the contributions to the fluctuations coming from the bulk and the edge of the droplet. The argument relies on a technical estimate (Lemma~\ref{lem:Abd}) which is proved afterwards. 

\begin{proposition}\label{lem:deco2}Let $\delta>0$.
Let $f\in L^{\infty}( \C, \R)$ and decompose $f=f_1+f_2+f_3$ where  $f_1,f_2,f_3$ are bounded (uniformly),  
$\dist(\supp f_1;\supp f_2) \ge \delta$ and $\supp f_1 \subset \{V\le \mu-2\delta\}$.
Then, for any $\alpha>0$,
\[\begin{aligned}
\Upsilon(f;\Pi_N) = \Upsilon(f_1+f_3;\Pi_N)+ \Upsilon(f_2+f_3;\Pi_N) -
\Upsilon(f_3;\Pi_N) + \O_{N\to +\infty}(e^{- \alpha \sqrt{N}}). 
\end{aligned}\]
\end{proposition}

\begin{proof}
Let $f_{\tau} :=\tau_1f_1+\tau_2 f_2+ f_3$ 
for  $\tau =(\tau_1, \tau_2)\in[0,1]^2$ and  $A_N^\tau :=
e^{f_\tau}\Pi_N(1+\Pi_N (e^{f_\tau}-1)\Pi_N)^{-1}\Pi_N$. By
Lemma~\ref{lem:Abd} ($\alpha=0$), these operators are well-defined and
bounded, uniformly with respect to $\tau\in[0,1]^2$ and $N\ge 1$.  
We compute for $\tau\in(0,1)^2$,
\[\begin{aligned}
\partial_{\tau_1} \Upsilon(f_\tau;\Pi_N)  & =  \tr(\Pi_N f_1 (1-\Pi_N)
A_N^\tau)  \\ 
\partial_{\tau_2}\partial_{\tau_1} \Upsilon(f_\tau;\Pi_N)  & =
\tr(\Pi_N f_1 (1-\Pi_N) \partial_{\tau_2} A_N^\tau)  
= \tr(\Pi_N f_1 (1-\Pi_N) (1-A_N^\tau)f_2 A_N^\tau )
\end{aligned}\]
using that $\partial_{\tau_2} A_N^\tau =  f_2 A_N^\tau -  A_N^\tau f_2  A_N^\tau$. 
By construction, $\Pi_N  A_N^\tau = A_N^\tau \Pi_N =\Pi_N$ so that
\[
(1-\Pi_N)(1-A_N^\tau) = (1-A_N^\tau)  
\]
and, using that $f_1f_2=0$, we can simplify the previous formula:
\[
\partial_{\tau_2}\partial_{\tau_1} \Upsilon(f_\tau;\Pi_N)  = \tr(\Pi_N
f_1 (1-A_N^\tau)f_2 A_N^\tau ) 
= -  \tr(\Pi_N f_1  A_N^\tau f_2 A_N^\tau ) . 
\]
Let $\rho : \C \to[0,1] $ be a 1-Lipschitz function,  such that
\[\begin{cases}
\rho=\delta &\text{on }\supp f_1 \\
\rho=0 &\text{on }\supp f_2 \cup \{V\le \mu-\delta\}.
\end{cases}\]
This is possible since, by assumptions, $\dist(\supp f_1;\supp f_2)
\ge \delta$ and $\supp f_1 \subset \{V\le \mu-2\delta\}$.

Let now $A_N^{\tau,\alpha}:=e^{\alpha\rho/\varepsilon_N} A_N^\tau e^{-\alpha\rho/\varepsilon_N}$.
In particular, since $f_2 e^{\alpha\rho/\varepsilon_N} = f_2$ and $f_1
e^{-\alpha\rho/\varepsilon_N} = f_1 e^{-\alpha\delta/\varepsilon_N}$,
we obtain 
\[\begin{aligned}
\tr(\Pi_N f_1  A_N^\tau f_2 A_N^\tau ) & = e^{-\alpha\delta/\varepsilon_N}  \tr(\Pi_N f_1  A_N^{\tau,\alpha} f_2 A_N^\tau ) \\
|\tr(\Pi_N f_1  A_N^\tau f_2 A_N^\tau ) | &\le  C\mathcal{N} e^{-\alpha\delta/\varepsilon_N} \|A_N^{\tau,\alpha}\| \|A_N^\tau\| 
\end{aligned}\]
using that $\|\Pi_N\|_{\tr} = \tr(\Pi_N)= \mathcal{N}$ and  $\|f_j\|\le C $ for some constant and $j\in\{1,2,3\}$. 

Then, by Proposition~\ref{prop:Pi_outside+Weyl}, $\mathcal{N} \leq C(\mu) N$, and by Lemma~\ref{lem:Abd}, 
$\|A_N^{\tau,\alpha}\| , \|A_N^\tau\| =\O_{N\to +\infty}(1)$, so we conclude that 
\[
|\partial_{\tau_2}\partial_{\tau_1} \Upsilon(f_\tau;\Pi_N) |\leq C N e^{-\alpha\delta/\varepsilon_N}.
\]

Finally, integrating the previous estimates,  
\[
\Upsilon(f_{(1,1)};\Pi_N)    =  \Upsilon(f_{(1,0)};\Pi_N) +  \Upsilon(f_{(0,1)};\Pi_N) -  \Upsilon(f_{0,0};\Pi_N) +\O( e^{-\alpha\delta/\varepsilon_N}) .
\]
Setting $f_{(1,1)}=f$, $f_{(1,0)}=f_1+f_3$, $f_{(0,1)}=f_2+f_3$ and
$f_{(0,0)}=f_3$, and replacing $\alpha$ with $\alpha/\delta$ (recall
$\delta$ is fixed) concludes the proof.
\end{proof}

\begin{lemma} \label{lem:Abd}
Let $g:\C\to\R$ be a bounded function and let $A_N := e^g \Pi_N(1+
\Pi_N(e^{g}-1)\Pi_N)^{-1}\Pi_N$. This operator is well-defined since
\[
  1+\Pi_N(e^g-1)\Pi_N=1-\Pi_N+\Pi_Ne^g\Pi_N\geq
  \min(e^{\min(g)},1){\rm Id}.\]
Let $\delta>0$. There exists ${\rm c}>0$ such that, for every
$1$-Lipschitz function $\rho : \C \to[0,1] $ supported in  $\{V\leq \mu-\delta\}$,
the operator $A_N^\alpha=e^{\alpha\rho/\varepsilon_N}
A_Ne^{-\alpha\rho/\varepsilon_N}$ is bounded and its operator norm is
controlled uniformly for $\alpha\in [0,c]$ and $N\geq 1$.
\end{lemma}

\begin{proof}
Let $Q$ be a self-adjoint operator such that  $1+ Q\ge  {\rm c} >0$ and $Q^\alpha= e^{\alpha\rho/\varepsilon_N}Qe^{-\alpha\rho/\varepsilon_N}$ satisfies $\|Q- Q^\alpha\| \le {\rm c}/2$. Then we have $e^{\alpha\rho/\varepsilon_N}(1+Q)^{-1}e^{-\alpha\rho/\varepsilon_N}=(1+Q^\alpha)^{-1} $
and, by the resolvent formula,
\[
(1+Q^\alpha)^{-1}  = (1+Q)^{-1}  + (1+Q^\alpha)^{-1}(Q-Q^\alpha)(1+Q)^{-1}
\]
so that, by assumptions, 
\[
\|(1+Q^\alpha)^{-1}\| \le \frac{\|(1+Q)^{-1}\|}{1-\|Q-Q^\alpha\|\|(1+Q)^{-1}\|} \le 2{\rm c}^{-1} . 
\]

Applying this to $Q=\Pi_N(e^{g}-1)\Pi_N$,  
$Q$ is self-adjoint (since $e^g>0$ on $\C$) with $ 1+ Q> e^{-C}$ where $C=\|g\|$ and we have
\[\begin{aligned}
e^{\alpha\rho/\varepsilon_N}Qe^{-\alpha\rho/\varepsilon_N}-Q  & = \Pi_N^\alpha(e^{g}-1)\Pi_N^\alpha - \Pi_N(e^{g}-1)\Pi_N \\
&= (\Pi_N^\alpha-\Pi_N)(e^{g}-1)\Pi_N^\alpha + \Pi_N(e^{g}-1)(\Pi_N^\alpha-\Pi_N) . 
\end{aligned}\]
Then by Lemma~\ref{lem:pertPi}, if $\alpha$ is sufficiently small,   
\[\begin{aligned}
\|e^{\alpha\rho/\varepsilon_N}Qe^{-\alpha\rho/\varepsilon_N}-Q \|
&\le e^C \|\Pi_N^\alpha-\Pi_N\| ( \|\Pi_N^\alpha\| +1) \\
&\leq C' \alpha .
\end{aligned}\]
Hence choosing $\alpha$ sufficiently small compared to $e^{-C}$, we obtain 
\[
\|
e^{\alpha\rho/\varepsilon_N}(1+\Pi_N(e^{g}-1)\Pi_N)^{-1}e^{-\alpha\rho/\varepsilon_N}\|
\leq C'' . 
\]
This implies that
\[
A_N^\alpha= e^g \Pi_N^\alpha e^{\alpha\rho/\varepsilon_N}(1+\Pi_N(e^{g}-1)\Pi_N)^{-1}e^{-\alpha\rho/\varepsilon_N}
\Pi_N^\alpha 
\]
satisfies, by Lemma~\ref{lem:pertPi} again (the operator $\Pi_N^\alpha$ is bounded independently of $N$), 
\[
\|A_N^\alpha\| \lesssim e^C \|\Pi_N^\alpha\|^2  \lesssim 1 . \qedhere
\]
\end{proof}

Finally we state a direct consequence of these decorrelation estimates that will be instrumental for our proof of the CLT. 

\begin{corollary}\label{cor:deco}
Let $\delta>0$ and, for $j\in\{1,2\}$, let $\chi_j : \C \to[0,1]$ be smooth functions such that 
\[
\1\{V\le \mu-\delta\} \le \chi_2 \le \1\{V\le \mu-\tfrac12\delta\}
\qquad\text{and}\qquad
\1\{|V- \mu| \le 2\delta\}\le  \chi_1 \le \1\{|V- \mu| \le 3\delta\}.
\]
Let $f : \C \to \R$ be a smooth bounded function, and let ${\rm f}_j:= f \chi_j$ for $j\in\{1,2,3\}$ where \[\chi_3 :=(\chi_1+\chi_2-1)\1\{V<\mu\}.\] 
Then, for every $\alpha>0$,
\[\begin{aligned}
\Upsilon(f;\Pi_N) = \Upsilon({\rm f}_1;\Pi_N)+ \Upsilon({\rm f}_2;\Pi_N) - \Upsilon({\rm f}_3;\Pi_N) + \O(e^{-\alpha \sqrt{N}}). 
\end{aligned}\]
By construction, ${\rm f}_1 ,{\rm f}_3$ are smooth, supported in the bulk $\{V\le \mu-\tfrac12\delta\} $, and ${\rm f}_2$ is smooth, supported in a neighborhood of the boundary $\{|V- \mu| \le 3\delta\}$. 
\end{corollary}

\begin{proof}
By construction, the following conditions hold; 
\[\begin{cases}
\1\{V\le \mu-3\delta\}  \le \varkappa_1  := \1\{V<\mu\}(1-\chi_1) \le \1\{V\le \mu-2\delta\}  \\
\1\{-\tfrac12\delta\le  V-\mu \le 2\delta\}  \le \varkappa_2 := \1\{V<\mu\} (1-\chi_2) + \1\{V\ge \mu\}\chi_1   \le \1\{-\delta\le V-\mu \le 3\delta\}  \\
\1\{ \delta\le \mu-  V \le 2\delta\} \le \varkappa_3 :=(\chi_1+\chi_2-1)\1\{V<\mu\} \le \1\{\tfrac12\delta\le \mu-V\le 3\delta\}  .
\end{cases}\]
Moreover, $\varkappa_j$ are smooth for $j\in\{1,2,3\}$ with $1-(\varkappa_1+\varkappa_2+\varkappa_3) = \1\{V\ge \mu\}(1-\chi_1)   $.

Now, let $f_j:= f \varkappa_j$ for $j\in\{1,2,3\}$ so that  $f=f_0+f_1+f_2+f_3$ where $f_0$ is supported in the forbidden region $\{V>\mu+2\delta\}$. Moreover, if $\|f\| \le C$ for some constant, then $\|f_j\| \le C$ for $j\in\{0,1,2,3,4\}$. 

First, applying Lemma~\ref{lem:deco1}, we obtain, for every $\alpha>0$,
\[
\Upsilon(f;\Pi_N) = \Upsilon(f_1+f_2+f_3;\Pi_N) + \O(e^{-\alpha\sqrt{N}}). 
\]
Second, applying Proposition~\ref{lem:deco2}, we conclude that 
\[
\Upsilon(f;\Pi_N) = \Upsilon(f_1+f_3;\Pi_N)+ \Upsilon(f_2+f_3;\Pi_N) - \Upsilon(f_3;\Pi_N) + \O(e^{-\alpha \sqrt{N}}). 
\]
By construction, ${\rm f}_1= f_1+f_3$, ${\rm f}_2= f_2+f_3$ and ${\rm f}_3= f_3$, this concludes the proof. 
\end{proof}

\section{Edge asymptotics}\label{sec:edge}
The goal of this section is to  prove Theorem~\ref{thm:CLT} when $f$ is smooth and supported near $\{V=\mu\}$. 
Throughout this section, we let $f:\C\to\R$ be a smooth function supported in $\{|V-\mu| \le \delta/2\}$ for a small $\delta>0$. Define the projection $\mathcal{X}_N : = \1\{\mu-\delta< H_N \le \mu+\delta\}$ and the operator 
\[
\mathcal A := \mathcal{X}_N(e^f-1)\mathcal{X}_N . 
\]
The operator $\1+\mathcal A$ is positive, bounded, thus $\log(\1+\mathcal A)$ is well-defined through functional calculus and we will consider the (modified) log-determinant\footnote{Since $\Pi_N$ has finite rank, this Fredholm determinant is well-defined and $\log\det(\1+\Pi_N\mathcal A \Pi_N) = \tr(\log(1+\Pi_N\mathcal A \Pi_N))$.} 
\[\begin{aligned}
\Gamma(\mathcal A )& : = \log\det(\1+\Pi_N\mathcal A  \Pi_N) -\tr(\Pi_N\log(\1+\mathcal A)\Pi_N) \\
&= \tr\big[ \log(\1+\Pi_N\mathcal A \Pi_N) - \Pi_N\log(\1+\mathcal A)\Pi_N \big] . 
\end{aligned}\]

In fact, one can replace $\log(1+\cdot)$ in \eqref{Gamma1} by any smooth function $\vartheta :\R \to\R$, with compact support, such that $\vartheta=\log(1+\cdot)$ on the set $(e^f-1)(\R)$. 
Moreover, by using the resolvent formula, one has \[(1-\Pi_N)\log(\1+\Pi_N\mathcal A \Pi_N) = 0\] so that 
\begin{equation} \label{Gamma1}
\Gamma(\mathcal A ) =  \tr\big[\Pi_N\big(\vartheta(\1+\Pi_N\mathcal A \Pi_N) - \vartheta(\1+\mathcal A)\big)\Pi_N \big] .
\end{equation}

We are interested in the asymptotics of the quantity $\Gamma(\mathcal A)$ as $N\to\infty$.  
We will prove the two following results.

\begin{proposition}\label{prop:Szego} 
$\Gamma(\mathcal A ) \to \tfrac12\Sigma^1_{\mathcal{D}}(f)$  as $N\to\infty$, where $\Sigma^1_{\mathcal{D}}(f) = \sum_{n =1}^\infty n |\widehat{f}_n|^2$
according to \eqref{var}--\eqref{Fcoeff}. 
\end{proposition}

\begin{proposition}\label{prop:trscexp}
Let $g:=e^f-1$.
Let $\vartheta :\R\to\R$ be a smooth function, with compact support, such that $\vartheta=\log(1+\cdot)$ on the range $g(\R)$.
One has as $N\to\infty$, 
\[
\Upsilon(f,\Pi_N) = \Gamma(\mathcal A) 
+ \tr\big(\Pi_N \big(\vartheta(\mathcal A) - \vartheta(g)\big)  \Pi_N\big)+o(1)
\]
and 
\[
\tr\big(\Pi_N \big(\vartheta(\mathcal A) - \vartheta(g)\big)  \Pi_N\big) 
\simeq \tfrac14  \int_{\mathcal{D}}   |\nabla f|^2\, \bgamma(\d x) = \tfrac14 \Sigma^2_{\mathcal{D}}(f) . 
\]
\end{proposition}

These asymptotics directly imply Theorem~\ref{thm:CLT} in the special case where the test function $f$ is $C^\infty_0$ and supported in a small neighborhood of the droplet boundary $\{V=\mu\}$.
Indeed, assembling the two last propositions, we obtain,  as $N\to\infty$,
\[
\Upsilon(f,\Pi_N)   \to \tfrac12 \big( \Sigma^2_{\mathcal{D}}(f)+\Sigma^1_{\mathcal{D}}(f) \big) = \tfrac12 \Sigma(f). 
\]

The rest of this section is mainly devoted to the proof of Proposition~\ref{prop:Szego} and it is organized as follows.

\begin{itemize}[leftmargin=*] \setlength\itemsep{0em}
\item Introduce
\[
\mathcal U_t(\mathcal A)  =  e^{it(\Pi_N\mathcal A\Pi_N)} - \Pi_N e^{it \mathcal A} \Pi_N , \qquad t\in\R.  
\]By \eqref{Gamma1} and the Fourier inversion formula for the Schwartz
function $\vartheta$, one has
\begin{equation} \label{Gamma2}
\begin{aligned}
\Gamma(\mathcal A )
=  \int_\R \widehat\vartheta(t) \tr[\Pi_N \mathcal U_t(\mathcal A)] \d t . 
\end{aligned}
\end{equation}

\item In Section~\ref{sec:RP}, we give some criteria on operators
  $A,B$ to replace $\mathcal U_t(A)$ with $\mathcal U_t(B)$ in formula \eqref{Gamma2}; see Proposition~\ref{prop:replace}.  
\item In Section~\ref{sec:Toeplitz}, we show that we can take $B$ to
be an infinite Toeplitz matrix with a smooth symbol coming from the function $e^f$ evaluated along the Hamiltonian flow restricted to the curve $\{V=\mu\}$. 
In particular, we verify that these matrices $A,B$ satisfy the hypothesis of Proposition~\ref{prop:replace}. 
\item In Section~\ref{sec:Sz}, we complete the proof of
  Proposition~\ref{prop:Szego}; the asymptotics of $\Gamma(B)$ follow from the  strong Szeg\H{o} limit theorem.
\item Finally, in Section~\ref{sec:scexp}, we give the proof of Proposition~\ref{prop:trscexp}.
\end{itemize}

\subsection{Replacement principle.} \label{sec:RP}
The goal of this section is to prove the following general proposition.

\begin{proposition} \label{prop:replace}
Let $\varepsilon>0$, let $A, B$ be two bounded self-adjoint operators
on a separable Hilbert space, and $\Pi$ be a projection let ${\rm C}>0$. 
Define for $t\in\R$, 
\[
\mathcal U_t(A) : =  e^{it(\Pi A\Pi)} - \Pi e^{it A}\Pi  . 
\]
Suppose that the following conditions hold:
\[
1)\quad \|\Pi B(1-\Pi)\|_{\tr} \le {\rm C}\, ,
\qquad
2)\quad \|[A,B]\| \le \varepsilon\, ,
\qquad
3)\quad
\| \Pi(A-B) (1-\Pi) \|_{\rm tr} \le \varepsilon \, ,
\]
\begin{equation} \label{cond4}
4)\quad \| \Pi B(1-\Pi)  (A-B) \|_{\rm tr}\leq \varepsilon\,,  \qquad  \|(A-B)\Pi B(1-\Pi)\|_{\tr} \le \varepsilon   \, .
\end{equation}
Then,  for any $t\in\R$, 
\[
\|\mathcal U_t(A)- \mathcal U_t(B)\|_{\tr}  \leq  \varepsilon |t| (1+ |t|+{\rm C}t^2 )  .
\]
\end{proposition}
The proof of this proposition is postponed to the end of this section,
after we prove a few preliminary estimates.

Let $t\in\R_+ \mapsto a(t)$ and $t\in\R_+ \mapsto r(t)$ be smooth
bounded operator-valued functions. Assume that $a(t)$ is self-adjoint
and consider the solution of the time dependent problem
\begin{equation}\label{Gron}
\begin{cases}
- i u'(t) = a(t) u(t) + r(t) , \quad t>0 \\
u(0)=0.
\end{cases}
\end{equation}
The solution $t\in\R_+ \mapsto u(t)$ is smooth, bounded and it satisfies (Duhamel's formula)
\[
u(\tau) = i \int_0^\tau e^{i \int_t^\tau a } r(t) \d t , \qquad\text{for }\tau\ge 0. 
\]
Let $\|\cdot\|_\bullet$ be an (operator) norm such that $\|AB\|_\bullet \le \|A\| \|B\|_\bullet$.
Since $a$ is self-adjoint (so that $t\in\R_+ \mapsto e^{i a(t)}$ is a unitary flow), we can bound for $\tau\ge 0$
\begin{equation} \label{Gronlin}
\|u(\tau) \|_\bullet  \le  \int_0^\tau  \|r(t)\|_\bullet \d  t .
\end{equation} 
Observe that the same bound holds if one considers the equation $- i u'(t) =  u(t)a(t) + r(t)$
instead of \eqref{Gron}. 

\begin{lemma} \label{lem:Gron}
Let $A, B$ be bounded self-adjoint operators. 
One has for any $t\ge 0$, 
\[
\| [ e^{itA} , B] \|_\bullet  \le  \|[A,B]\|_\bullet \,  t  . 
\]
In addition, if $ \|[A,B]\| \le \varepsilon$ and $\Delta$ is a trace-class operator satisfying 
$\|\Delta\|_{\tr} \le {\rm C}$ and $\|\Delta (A-B) \|_{\rm tr} \le \varepsilon$,  then for any $t\ge 0$,
\[
\|\Delta(e^{it A} -e^{i t B}) \|_{\rm tr}    \le  \varepsilon t(1+ \tfrac{{\rm C}}{2} t)  . \qedhere
\]
\end{lemma}

\begin{proof}~\newline 1) One has, for $t\in\R$, 
\[
-i\partial_t  [e^{it A} , B]  =   [e^{itA} , B] A + e^{itA} [A,B]  .
\]
This equation for $t\mapsto [e^{it A} , B] $  is of the type \eqref{Gron} with $a(t)=A$ and $r(t) = e^{itA} [A,B]$, so that $\|r(t)\|_\bullet \le \|[A,B]\|_\bullet$ for $t\ge 0$.
Then, by \eqref{Gronlin}, for $ \tau \ge 0$, 
\[
\| [ e^{i\tau A} , B] \|_\bullet   
\le \|[A,B]\|_\bullet \tau  . 
\]
2) Similarly, one has for $t\in\R$, 
\[
-i\partial_t (e^{it A} -e^{it B})   =  (e^{it A} -e^{it B})B + e^{it A} (A-B) 
\]
so that
\[
-i\partial_t\big(\Delta (e^{it A} -e^{it B})\big)  =   \Delta (e^{it A} -e^{it B})B  + \underbrace{ \Delta (A-B) e^{it A}  -  \Delta [e^{it A}, B]}_{=r(t)} \, . 
\]
Consequently, by 1), 
\[
\|r(t)\|_{\tr} \le \|\Delta (A-B) \|_{\rm tr}  +  \|\Delta\|_{\tr} \|[e^{it A}, B]\| \le  \varepsilon(1+{\rm C} t) 
\qquad\text{for }t\ge 0, 
\]
so that,  by \eqref{Gronlin} again (with $a(t)=B$), we conclude that  for any $\tau \ge 0$,
\[
\|\Delta(e^{i\tau A} -e^{i\tau B}) \|_{\rm tr}    \le  \varepsilon \int_0^\tau (1+{\rm C} t)  \d t
\le\tau(1+ \tfrac{{\rm C}}{2}\tau)  . \qedhere
\]
\end{proof}

\begin{lemma} \label{lem:propU}
Let $B$ be a self-adjoint operator and let $\Pi$ be an orthogonal
projector on a separable Hilbert space so that $\|\Pi B(1-\Pi)\|_{\tr} \le {\rm C}$.
One has, for $t\ge 0$,
\[
\|\Pi\mathcal U_t(B)\|_{\rm tr} \le {\rm C} t.
\]
Moreover, if $\Delta$ is a bounded operator such that 
\begin{equation} \label{Rcond}
\| \Pi B(1-\Pi)\Delta\Pi\| \vee  \| \Pi \Delta (1-\Pi)B\Pi\| \vee \|[\Delta, B]\| \le \varepsilon  \, , \qquad 
\|\Pi\Delta \Pi B(1-\Pi)\|_{\tr} \le \varepsilon , 
\end{equation}
then for $t\ge 0$,
\[
\|\Pi\Delta\Pi\mathcal U_t(B) \|_{\tr}   \le \varepsilon (1+\tfrac 32
{\rm C}t).
\]
\end{lemma}

\begin{proof}~\newline
1) By definition,  $\mathcal U_0(B) = (1-\Pi)$  and for $t\in\R$, 
\begin{equation} \label{derU1}
\begin{aligned}
-i \partial_t\mathcal U_t(B) & = \Pi B\Pi  e^{it(\Pi B\Pi)} - \Pi B e^{it B}\Pi \\
&= \Pi B\Pi  \mathcal U_t(B)- \Pi B(1-\Pi)e^{it B}\Pi . 
\end{aligned}
\end{equation}
Since $\Pi\mathcal U_0(B) =0$, by \eqref{Gron}--\eqref{Gronlin}, this implies that for $\tau \ge 0$, 
\[
\|\Pi\mathcal U_\tau(B) \|_{\tr}  \le \int_0^\tau \| \Pi B(1-\Pi)e^{it B}\Pi \|_{\tr} \d t
\le {\rm C}  \tau .
\]

\noindent
2) Let $\Delta$ be a bounded operator. By \eqref{derU1} again, one has with $R_0:= \Pi \Delta \Pi B(1-\Pi)$, 
\[\begin{aligned}
-i \partial_t\big(\Pi\Delta \Pi\mathcal U_t(B) \big) & = \Pi \Delta \Pi B\Pi  \mathcal U_t(B)- \Pi\Delta \Pi B(1-\Pi)e^{it B}\Pi \\
&= \Pi B \big(\Pi\Delta \Pi\mathcal U_t(B) \big) + \underbrace{\Pi[\Delta\Pi, B\Pi ]\Pi\mathcal U_t(B)- R_0 e^{it B}\Pi}_{=r(t)}
\end{aligned}\]
If we decompose
\[\begin{aligned}
\Pi[\Delta\Pi, B\Pi ]&= \Pi \big(\Delta B - \Delta(1-\Pi)B - B\Delta + B(1-\Pi)\Delta\big)\Pi  
&\qquad R_1 := \Pi\Delta(1-\Pi)B\Pi \, , \\
&= \Pi [\Delta,B] \Pi- R_1 + R_2
&\qquad R_2 := \Pi B(1-\Pi)\Delta\Pi \, .
\end{aligned}\]
The conditions \eqref{Rcond} imply that $\|R_1\|, \|R_2\|, \|[\Delta,B]\| \le \varepsilon$ and $\|R_0\|_{\tr} \le \varepsilon$, 
so that by 1), 
\[
\|r(t)\|_{\tr} \le \varepsilon (3\|\Pi\mathcal U_t(B) \|_{\tr}+1)  \le \varepsilon(1+3{\rm C} t) 
\qquad\text{for }t\ge 0, 
\]
Hence, by \eqref{Gronlin}, we conclude that for $\tau\ge 0$
\[
\|\Pi\Delta\Pi\mathcal U_\tau(B) \|_{\tr}   \le \varepsilon \tau(1+
\tfrac 32 {\rm C}\tau )  . \qedhere
\]
\end{proof}

\begin{proof}[Proof of Proposition~\ref{prop:replace}]
By \eqref{derU1}, for $t\in\R$, 
\[
-i \partial_t\big(\mathcal U_t(A)- \mathcal U_t(B)\big)
= \Pi A\Pi  \big(\mathcal U_t(A) - \mathcal U_t(B)\big) - r(t)
\]
where
\[
r(t) =  \Pi B(1-\Pi)(e^{it A}-e^{i t B})\Pi - \Pi(A-B)\Pi \mathcal U_t(B) + \Pi(A-B) (1-\Pi)e^{it A}\Pi \, . 
\]

We assume that   $\|[A,B]\| \le \varepsilon$, $\|  \Pi B(1-\Pi) \|_{\tr}\le {\rm C}$ and  
\[
\|  \Pi B(1-\Pi)  (A-B) \|_{\rm tr}  \vee \|(A-B)\Pi B(1-\Pi)\|_{\tr} \le \varepsilon .
\]
By Lemma~\ref{lem:Gron} with $\Delta=\Pi B(1-\Pi)$, this implies that
\[
\|\Pi B(1-\Pi)(e^{it A}-e^{i t B})\Pi \|_{\tr} \leq \varepsilon t(1+
\tfrac {\rm C}2 t)  . 
\]
Moreover, by  Lemma~\ref{lem:propU} with $\Delta=(A-B)$ -- since $\|\Pi(A-B)(1-\Pi)\| \le \varepsilon$ and $[\Delta, B] =[A,B]$, the conditions \eqref{Rcond} holds -- we have
\[
\|\Pi(A-B)\Pi \mathcal U_t(B) \|_{\tr} \leq  \varepsilon t (1+ \tfrac 32{\rm C}t )  . 
\]
If in addition $\|\Pi(A-B) (1-\Pi)  \|_{\tr} \le \varepsilon$, by
combining these estimates, we obtain
\[
  \|r(t)\|_{\tr} \leq \varepsilon(1+ 2t+2{\rm C}t^2).
\]
Using that $\mathcal U_0(A) = \mathcal U_0(B)=1-\Pi$, by \eqref{Gron}--\eqref{Gronlin}, we conclude that for  $\tau\ge 0$, 
\[\begin{aligned}
\|\mathcal U_\tau(A)- \mathcal U_\tau(B)\|_{\tr}  \leq \varepsilon
\int_0^\tau (1+2t+2{\rm C} t^2) \d t 
\leq \varepsilon \tau (1+ \tau+{\rm C}\tau^2) .
\end{aligned}\]
The regime $t\le 0$ follows by the exact same argument and this completes the proof.
\end{proof}

\subsection{Approximate Toeplitz structure.} \label{sec:Toeplitz}

The goal of this section is to prove that, under Assumptions
\ref{ass:V}, for $g:\C\to\R$ smooth and supported near
$\{V=\mu\}$,  the  operator $\mathcal A = \mathcal{X}_Ng \mathcal{X}_N$
is \emph{approximately Toeplitz} in the eigenbasis of $H_N=P_N VP_N$.

To present the result, recall from \eqref{flow} the flow $\big(z_\theta^\lambda ; \theta\in[0,2\pi] \big)$ on $\{V=\lambda\}$ for $\lambda\in[\mu-\delta, \mu+\delta]$.
Then, let $I_\lambda:=  \bgamma(V<\lambda)  $ and define the \emph{Fourier coefficients}
\begin{equation} \label{Acoeff}
\widehat{a}_n(I_\lambda) =  \int_{[0,2\pi]}  \hspace{-.3cm} g(z_\theta^\lambda) e^{-i n\theta} \frac{\d\theta}{2\pi}  , \qquad z_0^\lambda \in \{V=\lambda\} , \quad n\in\Z. 
\end{equation}
By assumptions the map $\lambda \mapsto I_\lambda$ is a smooth homeomorphism taking values in $\R_+$ and, according to \eqref{Fcoeff}, one has $\widehat{g}_n= \widehat{a}_n(I_\mu)$.

\begin{proposition} \label{prop:Toeplitz}
Under Assumptions~\ref{ass:V}, let $\mathcal{X}_N : = \1\{\mu-\delta/2< H_N \le \mu+\delta/2\}$. 
Let $g:\C\to\R$ be smooth and supported in $\{|V-\mu| < \delta/2\}$ and let $\mathcal A : = \mathcal{X}_Ng \mathcal{X}_N$. 
Let $({\rm u}_k)$ be the  eigenfunctions of $H_N$ with non-decreasing eigenvalues $\lambda_k \le \mu+\delta$.
We consider the matrix coefficients 
\begin{equation} \label{Amatrix}
A_{j,k}  = \langle {\rm u}_j | \mathcal A\, {\rm u}_k \rangle
\qquad\text{ for } \lambda_k, \lambda_j\in [\mu-\tfrac\delta 2,\mu+\tfrac \delta 2] . 
\end{equation}
We extend $A$ onto a matrix on $\ell^2(\Z)$ by 0 otherwise.

One has the expansion, for every $j,k\in\Z$
\begin{equation}\label{Approx}
A_{j,k}  =  \widehat{a}_{j-k}(\tfrac{j+k}{2N})+ N^{-1}Q_{j,k}, \qquad
|Q_{j,k}|\leq  C(1+|j-k|)^{-\kappa-1}.
\end{equation}
Moreover, the Fourier coefficients \eqref{Acoeff} satisfy for all
$\lambda\in \R$ and $n\in\Z$;
$\widehat{a}_n(0) = 0$ if $| I-I_{\mu}| \geq \frac{\delta}{2}$ and
\begin{equation} \label{aLip}
|\widehat{a}_n(I) | \leq C (1+n)^{-\kappa-1} , \qquad  
|\widehat{a}_n(I)- \widehat{a}_n(I') |  \leq C |I-I'|(1+n)^{-\kappa-1} . 
\end{equation}
\end{proposition}
\begin{proof}
First observe that the Fourier coefficients \eqref{Acoeff} satisfy for any $\kappa >0$,
\begin{equation}\label{conda}
|\widehat{a}_n(I)| \leq C (1+|n|)^{-\kappa} , \qquad \widehat{a}_n(I_\lambda) = 0 \quad\text{if }\lambda \notin[\mu-\tfrac\delta2,\mu+\tfrac\delta2].
\end{equation}
The first condition follows from the smoothness of $g$ and the second from the support condition. 
Moreover, the map $\lambda \in \R  \mapsto  \widehat{a}_n(\lambda)$ is also smooth and 
$\mathrm{Lip}( \widehat{a}_n) \lesssim (1+|n|)^{-\infty}$. 
This establishes \eqref{aLip}.

Let $\mathcal{T}_N :=  \{k\in\Z : \mathcal{X}_N{\rm u}_k = {\rm
u}_k\} = \{k\in\Z:  \lambda_k\in [\mu-\frac\delta 2,\mu+\frac \delta 2] \}$ and $g:=e^f-1$. This is an interval around $\mathcal{N}$ of size proportional to $\delta N$. 
By Remark \ref{rk:specstab},  up to an exponentially small error, one can replace $V$ by a another potential $V_1\in C^{\infty}_b(\C,\R)$ while computing the non-trivial entry of the matrix $A$: 
\[
A_{j,k}=\langle {\rm u}^1_j,P_N gP_N{\rm
u}^1_k\rangle+O(N^{-\infty}) \qquad \forall j,k\in
\mathcal{T}_N.
\]
Then, we can apply Theorem 13.10 of \cite{zworski_semiclassical_2012}: after conjugation by a unitary map
(the Fourier-Bros-Iagolnitzer transform), $H_N^1$ and $P_NgP_N$ are $O(N^{-\infty})$
close to pseudodifferential operators with symbol in the Dimassi-Sjöstrand class and the symbol of the conjugation
of $P_NgP_N$ is still supported in a small neighbourhood of
$\{V=\mu\}$ (viewed as a subset in $\R^2$).
This is the exact setting of Proposition 2.11 in \cite{deleporte_central_2023}, which yields \eqref{Approx} for all indices $j,k\in
\mathcal{T}_N$ (the errors $\O(N^{-\infty})$ are absorbed in $Q$, and
$I_{\lambda_k}=k+\O(1)$ for $k\in \mathcal{T}_N$). 

Finally, by construction, if $j \notin \mathcal{T}_N $ or $k\notin \mathcal{T}_N$, 
\[
A_{j,k} =0 , \qquad  Q_{j,k}=-N\widehat{a}_{j-k}(\tfrac{j+k}{2N})
\]
Using the conditions \eqref{conda}, for any $\kappa >0$, 
\(
\widehat{a}_{j-k}(\tfrac{j+k}{2N}) \lesssim N^{-\kappa}  (1+|j-k|)^{-\kappa},
\)
if $j \notin \mathcal{T}_N $ or $k\notin \mathcal{T}_N$. This completes the proof.
\end{proof}

We use Proposition~\ref{prop:Toeplitz} to prove the following estimates (these bounds are relevant to apply Proposition~\ref{prop:replace}). 

\begin{proposition} \label{prop:Toep_est}
Let $A, Q$ be the matrices on $\ell^2(\Z)$ defined in Proposition~\ref{prop:Toeplitz}  and let $B$ be the Toeplitz matrix 
\begin{equation} \label{TmatrixB}
B_{j,k}=\hat{a}_{j-k}(I_\mu) , \qquad j,k \in \Z. 
\end{equation}
Let also  $\Pi = \pi_{(-\infty, \mathcal N]}$ be the projection on $\operatorname{span}({\rm e}_k , k\le \mathcal N)$.
There exists $C>0$ such that
\[
0) \quad \|A\| , \|Q\|\leq C,
\qquad
1) \quad\|\pi_N  B(1-\pi_N)\|_{\tr}  \leq C \, ,
\]
\[
2)\quad \|[A,B]\| \leq  C N^{-1}\, ,
\qquad
3)\quad
\| \pi_N(A-B) (1-\pi_N ) \|_{\rm tr} \leq  CN^{-1} \, ,
\]
\[
4)\quad \| \pi_N B(1-\pi_N)  (A-B) \|_{\rm tr}  \vee \|(A-B)\pi_N
B(1-\pi_N)\|_{\tr} \leq C N^{-1} \, . 
\]
\end{proposition}

\begin{proof}
We denote $\varepsilon = N^{-1}$.
We will use the following norm estimates: if $M$ is a bounded matrix, 
\begin{equation} \label{normest}
\|M\|_{\tr} \le \sum_{j,k\ge 0} |M_{j,k}| , \qquad \|M\|^2  \le \sup_{k\in \N_0} \bigg( \sum_{i,j\ge 0} | M_{k,j}|  | M_{i,k}| \bigg) . 
\end{equation}
In this prove we use the notation $\lesssim$ to indicate an upper
bound up to a multiplicative constant which is independent on $N$.

\noindent
0) Since $A$ is self-adjoint, by \eqref{Approx}--\eqref{aLip}, 
\begin{equation*}  
\|A\| \le \sup_{k\in \Z} \bigg( \sum_{j\in \Z} | A_{k,j}| \bigg) \lesssim \sum_{n\in\Z} \frac1{1+|n|^\kappa} \lesssim 1. 
\end{equation*}
Along the same lines, we obtain
\[
\|Q\|\lesssim 1.
\]
\\
\noindent
1) $B$ is a Toeplitz matrix with a smooth symbol $b$ satisfying $|b_n| \lesssim (1+|n|^{\kappa+1})^{-1}$. In particular, $\|B\| \lesssim 1$
and, by \eqref{normest},
\[
\|\pi_N B(1-\pi_N)\|_{\tr} \le \sum_{ j< \mathcal{N} \le k} |B_{j,k}| \lesssim \sum_{-\infty<j< \mathcal{N}} \frac{1}{1+|j-\mathcal{N}|^\kappa} \lesssim 1 . 
\]
By \eqref{Approx}, one also has,
along the same lines,
\[
\|\pi_NQ(1-\pi_N)\|_{\tr}\lesssim 1.
\]

\noindent
2) Since $\|[Q,B]\| \le 2\|B\|\|Q\| \lesssim 1$, we have  $[A,B] = [A^0,B] +O(\varepsilon)$ where $A^0_{j,k} = \widehat{a}_{j-k}(\tfrac{j+k}{2N})$.
Moreover, since $B$ is a Toeplitz matrix, we compute for  $j, k\in \Z$, 
\[
[A^0,B]_{j,k}  = \sum_{n} b_n A^0_{j,k+n}- \sum_{n} b_nA^0_{j-n,k}.
\]
Using the Lipchitz-continuity \eqref{aLip} , with $\ell =j-k$, 
\[
|A^0_{j,k+n} - A^0_{j-n,k}|= | \widehat{a}_{\ell -n}(\tfrac{j+k}{2N}+\tfrac{n}{2N}) - \widehat{a}_{\ell -n}(\tfrac{j+k}{2N}-\tfrac{n}{2N}) | \lesssim \frac{n/N}{1+|n-\ell|^\kappa}
\]
where this estimate holds uniformly for $(n,\ell) \in \Z^2$.
Then
\[
\sup_{k\in\N_0}\bigg(  \sum_{j} \sum_{n} |b_n||A^0_{j,k+n} -A^0_{j-n,k} | \bigg) \lesssim 
\varepsilon \sum_{n\in\Z} \frac{1}{1+|n|^{\kappa}}   \sum_{j\in\Z} \frac{1}{1+|j|^{\kappa+1}}
\lesssim \varepsilon .
\]

\noindent
3) Recall from 1) that $\|\pi_N Q(1-\pi_N)\|_{\tr} \lesssim 1$, so it suffices to show that $| \pi_N(A^0-B) (1-\pi_N) \|_{\rm tr} \lesssim \varepsilon$.
Then, using again the Lipchitz-continuity \eqref{aLip}, together with
the fact that $\frac{\mathcal{N}}{N}=\mu+\O(N^{-1})$, we obtain,
for $ j <k$, 
\begin{equation} \label{ALip}
|(A^0-B)_{j,k}| \lesssim   \frac{1+ |j+k -2 \mathcal N|}{N(1+|k-j|^{\kappa+1})} \le \frac{\varepsilon}{1+|k-j|^{\kappa}} + \frac{2 \varepsilon|\mathcal{N} -j|}{1+|k-j|^{\kappa+1}} .   
\end{equation}
Then, by \eqref{normest},
\[\begin{aligned}
| \pi_N(A^0-B) (1-\pi_N) \|_{\rm tr} 
&\le \sum_{ j< \mathcal{N}\le k} |(A^0-B)_{j,k}| 
\lesssim \sum_{ j< \mathcal{N}}  \frac{ \varepsilon}{1+|\mathcal{N}-j|^{\kappa-1}}
\lesssim \varepsilon . 
\end{aligned}\]

\noindent
4) Finally, we prove \eqref{cond4}. 
By \eqref{ALip}
\[
\sum_{ k} |(A-B)_{n,k}| 
\lesssim \epsilon(1+ |n-\mathcal{N}|)
\]
so that by \eqref{normest}, 
\[\begin{aligned}
\| \pi_N B(1-\pi_N)  (A-B) \|_{\rm tr}  & \le  \sum_{ j<\mathcal{N}\le n} |B_{j,n}| \sum_{ k} |(A-B)_{n,k}| \\
&\lesssim \varepsilon \sum_{j<\mathcal{N}\le n}  \frac{1+ (n-\mathcal{N})}{1+|n-j|^{\kappa+1}} \\
&\lesssim \varepsilon \sum_{\ell \ge 0} \frac{1+ \ell}{1+\ell^{\kappa}} \lesssim \varepsilon .
\end{aligned}\]
Similarly, one has
\[\begin{aligned}
\|(A-B)\pi_N B(1-\pi_N)\|_{\tr} & \le  \sum_{ n<\mathcal{N}\le k} |B_{n,k}| \sum_{0\le j} |(A-B)_{j,n}| \\
&\lesssim \sum_{n\le \mathcal{N} } \frac{\epsilon(1+ |n-\mathcal{N}|)}{1+|\mathcal{N}-n|^{\kappa}} \lesssim \varepsilon .
\end{aligned}\]
This completes the proof. 
\end{proof}

\subsection{Szeg\H{o}-type limit theorem; Proof of Proposition~\ref{prop:Szego}.} \label{sec:Sz}

Recall that according to  \eqref{Gamma2},
\[
\Gamma(\mathcal A) =  \int_\R \widehat\vartheta(t) \tr[\Pi_N\mathcal U_t(\mathcal A) ] \d t
\]
where for any $\kappa >4$,  $|\widehat\vartheta(t)| \leq C (1+|t|^\kappa)^{-1}$ for $t\in\R$ (the function $\vartheta :\R \to \R$ is $C^\infty_0$). 
Recall that $\operatorname{span}({\rm u}_k; k\in [1,\mathcal N])$ is the range of $\Pi_N$, then completing $({\rm u}_k)$ in an appropriate orthonormal basis of $\mathcal B_N$, the operator $\mathcal A$ is (unitary) equivalent to the matrix $A$ on $\ell^2(\Z)$, \eqref{Amatrix} and $\Pi_N$ is equivalent to $\pi_{[1,\mathcal N]}$.
Then, in the above formula, one can replace
\[
\mathcal U_t(\mathcal A)  \leftarrow \mathcal U_t(A) =  e^{it(\Pi A \Pi)} - \Pi e^{it  A} \Pi
\]
with $\Pi=\pi_{(-\infty, \mathcal N]}$.
This is the case because $\Pi A \Pi=\pi_{[1,\mathcal N]} A \pi_{[1,\mathcal N]} = $ and  $ \Pi\pi_{[1,\mathcal N]} =\pi_{[1,\mathcal N]}$.

Consequently,
\[
\Gamma(\mathcal A) =  \int_\R \widehat\vartheta(t) \tr[\pi_{[1,\mathcal N]}\mathcal U_t(A) ] \d t
\]
and, Proposition~\ref{prop:replace}, one can replace $A$ by the matrix $B$, \eqref{TmatrixB}, in the above formula up to an error  $\varepsilon =\O(N^{-1})$. 
The relevant estimates are derived in Proposition~\ref{prop:Toep_est}. 
By  \eqref{Gamma1}, this implies that as $N\to\infty$,
\[\begin{aligned}
\Gamma(\mathcal A)  & =  \int_\R \widehat\vartheta(t) \tr[\pi_{[1,\mathcal N]}\mathcal U_t(B) ] \d t + \O(N^{-1})  \\
& = \log\det_{[1,\mathcal N]}(T) - \tr(\pi_{[1,\mathcal N]} L)+\O(N^{-1}) \, ,
\end{aligned}\]
using that $B=T-\1$ where $T= T(m)$ is a Toeplitz matrix
with a symbol $m: \theta \in [0,2\pi] \mapsto
e^f(z_\theta^\mu)$, and where $L = \log T $. 
In particular, $L$ is also a Toeplitz matrix with symbol  $\log(m)  : \theta \in [0,2\pi]  \mapsto  f(z_\theta^\mu)$.

In particular, as $N\to +\infty$,  \[\Gamma(\mathcal A ) \to\tfrac12 \Sigma^1_{\mathcal{D}}(f).\]

\newpage

\subsection{Proof of Proposition~\ref{prop:trscexp}.} \label{sec:scexp}

The proof is divided in several elementary propositions. Recall that 
$\Pi_N = \1\{H_N \le \mu\}$ and $\mathcal{X}_N = \1\{\mu-\delta< H_N \le \mu+\delta\}$.

\begin{lemma} \label{lem:locedge}
Let $f:\C\to\R$ be smooth, supported in $\{|V-\mu| < \delta/2\}$, and let  $\mathcal A = \mathcal{X}_N(e^f-1)\mathcal{X}_N$.
Under Assumption~\ref{ass:V} a), for any $\alpha>0$, one has
\[
\log\det(\1-\Pi_N+\Pi_N e^{f} \Pi_N)=  \log\det(1+\Pi_N\mathcal A  \Pi_N)+ \O( e^{-\alpha\sqrt{N}} ) . 
\]
\end{lemma}

\begin{proof} This is a direct consequence of Proposition~\ref{prop:decay_kernel} and the proof is analogous to that of Lemma~\ref{lem:deco1}. 
Let $g:=e^{f}-1$. One has  
\[
\det(1+\Pi_N g\Pi_N) = \operatorname{det}(\underbrace{1+\Pi_N\mathcal A  \Pi_N}_{=:B})  
\operatorname{det}(1+ B^{-1}(\underbrace{\Pi_N((1-\mathcal{X}_N)g+ \mathcal{X}_Ng(1-\mathcal{X}_N)\Pi_N)}_{=:R}) )
\]
Since $\Pi_N(1-\mathcal{X}_N) = (1-\mathcal{X}_N)\Pi_N = \1\{H_N \le\mu-\delta\}$, by \eqref{trextdecay}, for any $\alpha>0$, 
\[
\|R\|_{\tr}\lesssim e^{-\alpha \sqrt{N}} . 
\]
Since $B$ is a bounded positive operator $(e^{-\min f}\le B \le e^{\max f})$
\[
|\log \det(1+ B^{-1}R)| \le \|B^{-1}\| \|R\|_{\tr}\lesssim e^{-\alpha \sqrt{N}} . 
\]
This proves the claim. 
\end{proof}

By \eqref{Gamma1} and Lemma~\ref{lem:locedge}, since $\vartheta(g) = f$ and $\vartheta(\mathcal A) =\log(1+\mathcal A)$,  we obtain as $N\to\infty$
\[\begin{aligned}
\Upsilon(f,\Pi_N) 
&=   \log\det(1+\Pi_N\mathcal A  \Pi_N) -  \tr(\Pi_Nf\Pi_N) +  \O(
e^{-\alpha\sqrt{N}} ) \\
& = \Gamma(\mathcal A)+ \tr\big(\Pi_N \big(\vartheta(\mathcal A) - \vartheta(g)\big)  \Pi_N\big)  +  \O(
e^{-\alpha\sqrt{N}} ) .
\end{aligned}\]
This proves the first claim of Proposition~\ref{prop:trscexp}. 

\medskip

To obtain the second claim, we rely on a two-term
semiclassical expansion for the functional calculus of
Berezin--Toeplitz operators, which is a
consequence of a two-term expansion for composition.

\begin{lemma}\label{prop:BTcomp}
Given smooth functions $a, b : \C \to \R$, one has, in operator norm
\begin{equation}\label{BTcalc}
\begin{aligned}
P_NaP_NbP_N&=P_Nab P_N +\O(N^{-1}\|a\|_{C^1}\|b\|_{C^1})  ,\\
P_NaP_NbP_N&=P_N(ab-N^{-1}\partial a \overline{\partial}b)P_N +\O(N^{-3/2}\|a\|_{C^2}\|b\|_{C^2}).
\end{aligned}
\end{equation}
\end{lemma}
\begin{proof}
These estimates follow by writing $P_Na(1-P_N)bP_N = P_N[P_N,a][P_N,b]P_N$. Then, the first bound follow directly from Lemma~\ref{lem:semiclassic_com}. 
For the second bound, by a Taylor expansion, the commutator satisfies
\[\begin{aligned}
[P_N,a] : (x,z) & \mapsto  \partial a(x) (z-x)  P_N(x,z) + \overline{\partial} a(x)  (\overline{z-x})  P_N(x,z)  + O(N^{-1}\|a\|_{C^2}) \\
[P_N,b] : (z,w) &  \mapsto   P_N(z,w) (w-z) \partial b(w)+   P_N(z,w)  (\overline{w-z})   \overline{\partial} b(w) + O(N^{-1}\|b\|_{C^2})
\end{aligned}\]
where the errors hold in operator norm.
Then,  
\begin{align}
&\notag P_N[P_N,a][P_N,b]P_N \\
&\label{commmain} = - \int P_N(\cdot,x) \partial a(x)  P_N(x,z) (z-x)(\overline{z-w}) P_N(z,w) \overline{\partial} b(w) P_N(w,\cdot) \bgamma(\d x)\bgamma(\d z) \bgamma(\d w)   \\
&\label{commerr}\quad-  \int P_N(\cdot,x) \overline{\partial} a(x)  P_N(x,z)  (\overline{z-x})(z-w) P_N(z,w)  \partial b(w) P_N(w,\cdot)
\bgamma(\d x)\bgamma(\d z) \bgamma(\d w)  \\
&\quad\notag+\O(N^{-3/2}\|a\|_{C^2}\|b\|_{C^2}),
\end{align}
where the other two terms vanish exactly using the reproducing property of $P_N$.  
For the error term, there are three contributions, obviously a term $\O(N^{-2}\|a\|_{C^2}\|b\|_{C^2})$ and terms of the form 
\[
\|(x,z) \mapsto \partial a(x) (z-x)  P_N(x,z)\|  \| O(N^{-1}\|b\|_{C^2})  =O(N^{-3/2}\|a\|_{C^1}\|b\|_{C^2}). 
\]
by Lemma~\ref{lem:semiclassic_com}. 

Using that $ P_N(x,z) (x-z) P_N(z,w) =   N^{-1} \overline{\partial_z}\{ P_N(x,z) P_N(z,w) \} $, integrating by parts,  we obtain
\[\begin{aligned}
\eqref{commmain}& =  N^{-1} \int P_N(\cdot,x) \partial a(x)  \overline{\partial_z}\{ P_N(x,z) P_N(z,w) \}  (\overline{w-z})  \overline{\partial} b(w) P_N(w,\cdot)  \\
& =  N^{-1} \int P_N(\cdot,x) \partial a(x)   P_N(x,w)   \overline{\partial} b(w) P_N(w,\cdot) \\ 
& =   N^{-1} P_N \partial a P_N \overline{\partial} b P_N \\
& =  N^{-1} P_N \partial a  \overline{\partial} b P_N +  \O(N^{-2}\|a\|_{C^2}\|b\|_{C^2}).
\end{aligned}\] 
using the first bound \eqref{BTcalc} at the last step. 
Similarly, 
\[\begin{aligned}
\eqref{commerr}
& =  N^{-2} \int  \overline{\partial} a(x)  \partial_x\{P_N(\cdot,x) P_N(x,z)\} \overline{ \partial_w}\{P_N(z,w) P_N(w,\cdot)\} \partial b(w) \\
& = N^{-2} \int  P_N(\cdot,x) \Delta a(x) P_N(x,z) P_N(z,w) \Delta b(w)P_N(w,\cdot) \\
& =N^{-2}P_N \Delta a P_N \Delta b P_N.
\end{aligned}\] 
So \eqref{commerr} is part of the error term and this proves the
second  claim of \eqref{BTcalc}.
\end{proof}

\begin{lemma}\label{prop:scexp}
Let $g:\C\to\R$, smooth and bounded and let $\vartheta :\R \to\R$ be a smooth function with compact support. 
As $N\to\infty$, in operator norm,
\[
\vartheta(P_N g P_N) = P_N\big(\vartheta (g)- N^{-1} \vartheta''(g) |\partial g|^2\big) P_N + O(N^{-3/2}).
\]
\end{lemma}

\begin{proof}
We rely on the Helffer--Sjöstrand formula: given a test function $\vartheta\in C^{\infty}_c(\R)$, there exists $\widetilde{\vartheta}\in C^{\infty}_c(\C)$ such that for $\lambda,\eta \in \R$
\[
|\overline{\partial}\widetilde{\vartheta}(\lambda+i\eta)|=\O(|\eta|^{\infty}) \qquad \qquad \widetilde{\vartheta}(\lambda)=\vartheta(\lambda).
\]
$\widetilde{\vartheta}$ is called an almost holomorphic extension of
$\vartheta$.
Then, by Cauchy's formula, for $\lambda\in\R$, 
\[
\vartheta(\lambda)=\frac{1}{2i}\int_\C \frac{\overline{\partial}\widetilde{\vartheta}(z)}{\lambda-z}\bgamma(\d z) .
\]
By spectral theory, this formula extends to bounded self-adjoint operators; if $\mathcal{G} := P_N g P_N$ with $g :\C \to\R$ bounded, we have 
\begin{equation}\label{eq:HeSj}
\vartheta(\mathcal{G})
=\frac{1}{2i}\int_\C \overline{\partial}\widetilde{\vartheta}(z)(\mathcal{G}-z)^{-1}\bgamma(\d z) 
=\frac{1}{2i}\int_\C \overline{\partial}\widetilde{\vartheta}(z) (P_N(g-z)P_N)^{-1} \bgamma(\d z) 
\end{equation}
where $(P_N(g-z)^{-1}P_N)^{-1}$ is defined on $\mathcal{B}_N $. 
Then, for $z\in \C\setminus \R$, our task is to understand the operator $(P_N(g-z)P_N)^{-1}$ up to $\O(N^{-2})$, with an explicit control in $\Im(z)$.
To do so, we use the calculus for Berezin--Toeplitz quantization

Now, \eqref{BTcalc} allows us to produce an expansion of $[P_N(g-z)P_N]^{-1}$. 
Fix $z\in\mathbb{H}_+$, by \eqref{BTcalc}, we have
\[
P_N(g-z)P_N \left(\frac{1}{g-z}-N^{-1}\frac{|\partial g|^2}{(g-z)^3}\right)P_N
= P_N +\O\bigg(\frac{1}{N^2 \Im(z)^4}\bigg)
\]
Since $\|(P_N(g-z)^{-1}P_N)^{-1}\| \le \Im(z)^{-1}$ on $\mathcal{B}_N$, we obtain 
\[
(P_N(g-z)^{-1}P_N)^{-1} = P_N\left(\frac{1}{g-z}-N^{-1}\frac{|\partial g|^2}{(g-z)^3}\right)P_N + \O\bigg(\frac{1}{N^2 \Im(z)^5}\bigg) .
\]
The control of the error $\Im(z)^{-5}$ is not relevant, as  plugging this expansion in \eqref{eq:HeSj} we conclude that, in operator norm,
\[
\vartheta(\mathcal{G})=P_N\left(\frac{1}{2i}\int_{\C}\frac{\overline{\partial}\widetilde{\vartheta}(z)}{g-z}\bgamma(\d z)\right)P_N-\frac1N P_N\left(\frac{|\partial g|^2}{2i}\int_\C\frac{\overline{\partial}\widetilde{\vartheta}(z)}{(g-z)^3}\bgamma(\d z)\right)P_N+\O(N^{-2}).\]
The principal term is simply $\vartheta(g)$ and the second term, after integrating by parts twice, is $|\partial g|^2\vartheta''(g)$. This concludes the proof of the first part of the proposition.
\end{proof}

We are now ready to complete the proof of Proposition~\ref{prop:trscexp}. 

\begin{proposition}
Let $f:\C\to\R$ be smooth, supported in $\{|V-\mu| < \delta/2\}$, and let  $g:=e^f-1$.
Let $\mathcal{X}_N = \1\{\mu-\delta< H_N \le \mu+\delta\}$ and let  $\mathcal A = \mathcal{X}_N(e^f-1)\mathcal{X}_N$.
Let $\vartheta :\R\to\R$ be a smooth function, with compact support, such that $\vartheta=\log(1+\cdot)$ on the range $g(\R)$.
Then, as $N\to\infty$, 
\[
\tr\big(\Pi_N \big(\vartheta(\mathcal A) - \vartheta(g)\big)  \Pi_N\big) 
\to  \tfrac12 \Sigma^2_{\mathcal{D}}(f) . 
\]
\end{proposition}

\begin{proof}
Let $\mathcal{G} := P_N g P_N$. 
Both $ \mathcal{A},\mathcal{G}$ are bounded operators on the Hilbert space $\mathcal{B}_N$. 
Our first goal is to show that  for any $\alpha>0$,
\begin{equation} \label{Btrunc}
\| \mathcal{G}-\mathcal{A}\|_{\tr} \lesssim e^{-\alpha\sqrt{N}} .
\end{equation}
Write $P_N =\mathcal{X}_N + \mathcal{Q}_N $ where $\mathcal{Q}_N =\1\{H_N> \mu+\delta\} + \1\{H_N\le \mu-\delta\} $ and we decompose
\[\begin{aligned}
\mathcal{G} & = \mathcal{A}+ \mathcal{Q}_N g P_N +  \mathcal{X}_Ng\mathcal{Q}_N . 
\end{aligned}\]
In particular,  
\[
\| \mathcal{G}-\mathcal{A}\|_{\tr} \le \|\mathcal{Q}_N g \|_{\tr}+ \|g\mathcal{Q}_N\|_{\tr}. 
\]
Since $g:\C\to\R$ is bounded and supported in $\{|V-\mu|\le \delta/2\}$,  
by \eqref{trextdecay}, for any $\alpha>0$, 
\[\begin{aligned}
\| \1\{H_N\le \mu-\delta\}  g \|_{\tr} \vee \|g \1\{H_N\le \mu-\delta\} \|_{\tr}
& \le  \|g\| \int \1\{V(x)\ge \mu-\delta/2\} \1\{H_N \le \mu-\delta\}(x,x) \bgamma(\d x) \\
& \lesssim e^{-\alpha\sqrt N}
\end{aligned}\]
and similarly, by \eqref{extptwdecay}, 
\[\begin{aligned}
\| \1\{H_N > \mu+\delta\}  g \|_{\tr} \vee \|g \1\{H_N> \mu+\delta\} \|_{\tr} & \le \|g\| \int \1\{V(x)\le \mu+\delta/2\} \1\{H_N > \mu+\delta\}(x,x) \bgamma(\d x) \\
&\lesssim e^{-\alpha\sqrt{N}} .
\end{aligned}\]
Combining these estimates, this proves \eqref{Btrunc}.

\smallskip

Then, using the Helffer-Sjöstrand formula, \eqref{eq:HeSj}, 
\[
\vartheta(\mathcal{A})- \vartheta(\mathcal{G})
=\frac{1}{2i}\int_\C  \overline{\partial}\widetilde{\vartheta}(z) \big((\mathcal{A}-z)^{-1} - (\mathcal{G}-z)^{-1}\big)\bgamma(\d z) 
=\frac{1}{2i}\int_\C  \overline{\partial}\widetilde{\vartheta}(z) \big((\mathcal{A}-z)^{-1} (\mathcal{G}-\mathcal{A})(\mathcal{G}-z)^{-1}\big)\bgamma(\d z)  
\]
where, for any $\kappa>0$,  $|\overline{\partial}\widetilde{\vartheta}(\lambda+i\eta)| \le C_\kappa |\eta|^{\kappa}$ for $\lambda,\eta \in \R$. 
Thus, using that $\|  (\mathcal{A}-z)^{-1}\| \le \eta^{-1}$ and similarly for $\mathcal{G}$ (both operators are self-adjoint),  we have
\[
\|\vartheta(\mathcal{G})- \vartheta(\mathcal{A})\|_{\tr}
\le \| \mathcal{G}-\mathcal{A}\|_{\tr} \int_\C \frac{|\overline{\partial}\widetilde{\vartheta}(z)|}{(\Im z)^{-2}}  \bgamma(\d z)  
\lesssim \| \mathcal{G}-\mathcal{A}\|_{\tr} 
\]
and this quantity is exponentially small. 

\smallskip

By Lemma~\ref{prop:scexp},  using that $\| \Pi_N\|_{\tr} =N$, this implies that as $N\to\infty$, 
\[\begin{aligned}
\tr\big(\Pi_N \big(\vartheta(\mathcal A) - \vartheta(g)\big)  \Pi_N\big)  
&= \tr\big(\Pi_N \big(\vartheta(\mathcal{G}) - \vartheta(g)\big)  \Pi_N\big) + \O(e^{-\alpha\sqrt{N}})\\
&= N^{-1} \tr\big(\Pi_N\vartheta''(g) |\partial g|^2\Pi_N \big) + O(N^{-1/2}) .
\end{aligned}\]
Finally,  $\vartheta''(g) = - e^{-2f} $ since $\vartheta=\log(1+\cdot)$ on $g(\R)$ and $g=e^f-1$ ($|\partial g| = \tfrac12 |\nabla f| e^f$) so that by Proposition~\ref{prop:Weyl}, as $N\to\infty$, 
\[ 
N^{-1} \tr\big(\Pi_N\vartheta''(g) |\partial g|^2\Pi_N \big) \to  \frac14  \int_{\mathcal{D}}   |\nabla f|^2\, \bgamma(\d x) = \tfrac12 \Sigma^2_{\mathcal{D}}(f) . 
\]
This concludes the proof.
\end{proof}

\section{Bulk asymptotics} \label{sec:bulk}

\subsection{Generalities.}
The following  considerations holds in an abstract context. In this
section, we let $\mathfrak{X}$  be a  separable complete metric space and 
$\Q$ be a self-adjoint  projection on $L^2(\mathfrak{X})$.

\begin{proposition} \label{prop:Gauss}
Let $f: \mathfrak{X}\to\R$ be a map such that  the operator-valued function $ \lambda\in[0,1] \mapsto \Q (e^{\lambda f}-1)\Q$ is trace-class and  $1+ \Q (e^{\lambda f}-1)\Q >0$ for 
$\lambda\in[0,1]$. 
Then
\[
\sup_{\lambda\in[0,1]}|\Upsilon(\lambda f;\Q) - \lambda^2\Sigma/2 | \le \eta \Sigma(1+ \eta\Sigma) , \qquad\quad 
\Sigma=\Sigma(f;\Q) = - \tfrac12 \tr([f,\Q]^2) 
\] 
and
$\eta=\eta(f;\Q)= \sup_{\lambda\in[0,1]} \|(1-\Q)e^{\lambda f} \Q (1+ \Q (e^{\lambda f}-1)\Q)^{-1}\|$.
\end{proposition}

Note that in Proposition~\ref{prop:Gauss}, $\Sigma(f;\Q)\ge 0$ since it corresponds (up to a factor $\frac 12$) to the Hilbert-Schmidt norm of $[f,\Q]$ and the bound is meaningful only if $\Sigma(f;\Q) <\infty$. 
Moreover, by assumption, there is a constant $\delta>0$ so that $1+A_\lambda > \delta$ for 
$\lambda\in[0,1]$ so that
$ \eta(f;\Q) \le \delta^{-1}  \sup_{\lambda\in[0,1]} \|(1-\Q)e^{\lambda f} \Q\|$.

\begin{proof}
Throughout the proof, we denote for $\lambda\ge 0$,
\[
\g_\lambda := e^{\lambda f} , \qquad\qquad 
R_\lambda :=(1+\Q (\g_\lambda-1)\Q)^{-1} ,\qquad\qquad
B_\lambda :=\Q g(1-\Q)f_\lambda \Q. 
\]
By assumptions, $R_\lambda$ is bounded and we claim that
\begin{equation} \label{inv1}
\Q\g_\lambda \Q R_\lambda
=\Q+ (1- \Q)\g_\lambda \Q R_\lambda.
\end{equation}
This follows by writing
\[\begin{aligned}
\Q\g_\lambda \Q(1+\Q (\g_\lambda-1)\Q)^{-1} 
&=(\Q+ \Q(\g_\lambda-1)\Q)(1+\Q (\g_\lambda-1)\Q)^{-1}  \\
&= 1-(1-\Q)(1+\Q (\g_\lambda-1)\Q)^{-1} 
\end{aligned}\]
and multiplying both sides by $\Q$. 
Then, we compute using \eqref{inv1}, 
\[\begin{aligned}
\partial_\lambda \log\det(1+\Q (\g_\lambda-1)\Q) 
&=\tr\big(\Q f \g_\lambda\Q R_\lambda\big) \\
&=  \tr\big(\Q f\Q\big)
+   \tr\big(B_\lambda R_\lambda\big).
\end{aligned}\]
Equivalently (since $B_0 = 0$ because $\g_0=1$), 
\begin{equation} \label{der1}
\partial_\lambda  \Upsilon(\lambda f ;\Q) = \tr\big(B_\lambda R_\lambda\big) , \qquad \qquad \partial_\lambda  \Upsilon(\lambda f ;\Q)|_{\lambda=0}=0.
\end{equation}

We also compute 
\begin{equation}\label{diffQ}
\begin{aligned}
\partial_\lambda B_\lambda & =\Q f(1-\Q) f \g_\lambda \Q  \\
&= D \Q \g_\lambda \Q + D(1-\Q)\g_\lambda\Q   , \qquad 
D:= \Q f(1-\Q)f , 
\end{aligned}
\end{equation}
and using \eqref{inv1}, 
\begin{equation}\label{diffR}
\begin{aligned}
-\partial_\lambda R_\lambda  &= R_\lambda \Q f \g_\lambda \Q  R_\lambda
=  R_\lambda\Q f \Q + R_\lambda \Q f (1-\Q) \g_\lambda \Q  R_\lambda\\
&= R_\lambda\Q f\Q  + R_\lambda B_\lambda  R_\lambda .
\end{aligned}
\end{equation}

Then, using \eqref{diffQ}--\eqref{diffR} and \eqref{inv1} again, we obtain
\[\begin{aligned}
\partial_\lambda \tr (B_\lambda R_\lambda) & = \tr(D \Q \g_\lambda \Q R_\lambda) + \tr(D(1-\Q) \g_\lambda \Q R_\lambda)-\tr (B_\lambda \partial_\lambda R_\lambda) \\
&= \tr(D \Q) +  \tr(D(1-\Q)\g_\lambda \Q R_\lambda) -\tr(B_\lambda R_\lambda \Q f \Q)  
-\tr(B_\lambda R_\lambda B_\lambda  R_\lambda)
\\
&= \tr(D) +  \tr(\Q(D-D^*)(1-\Q)\g_\lambda \Q R_\lambda)  - \tr((B_\lambda R_\lambda)^2) 
\end{aligned}\]
where we used cyclicity of the trace and rewrote the third term as
\[
\tr(B_\lambda R_\lambda \Q f\Q) = \tr(\Q f\Q f(1-\Q)\g_\lambda \Q R_\lambda)
=\tr(\Q D^*(1-\Q)\g_\lambda \Q R_\lambda).
\]

To control the above terms, observe that since $\Q$ is a projection, 
\[\begin{aligned}
\|B_\lambda R_\lambda\|_{\rm H} &=\|\Q g(1-\Q)\g_\lambda \Q R_\lambda\|_{\rm H} \le \|(1-\Q)\g_\lambda \Q R_\lambda\| \cdot \|\Q g(1-\Q)\|_{\rm H} \\
\|\Q g(1-\Q)\|_{\rm H}  &=  \tr(\Q g(1-\Q)g\Q) =\tr(D) 
\end{aligned}\]
so that, with $\eta = \sup_{\lambda\in[0,1]} \|(1-\Q)\g_\lambda \Q R_\lambda\|$, 
\[\begin{aligned}
| \tr((B_\lambda R_\lambda)^2) |  &\le \|B_\lambda R_\lambda\|_{\rm H}^2 \le  
\tr(D)^2 \eta^2
\\ 
\|\Q D(1-\Q)\g_\lambda \Q R_\lambda\|_{\tr} &\le \|(1-\Q)\g_\lambda \Q R_\lambda\| \cdot \|D\|_{\tr}  = \tr(D) \eta
\end{aligned}\]
and similarly replacing $D$ by $D^*$.

Altogether, this implies that 
\[
|\partial_\lambda \tr (B_\lambda R_\lambda) -\tr(D)| \le  \eta\tr(D)(2+ \eta\tr(D)). 
\]
Then, going back to \eqref{der1}, we conclude that for $\lambda\in[0,1]$, 
\[
|\partial_\lambda^2 \Upsilon(\lambda f ;\Q)-\tr(D)|  \le  \eta\tr(D)(2+ \eta\tr(D)) .
\]
Thus, integrated this formula twice and using that
\[\begin{aligned}
\tr(-[f,\Q]^2) &= \tr(f\Q f + \Q f^2\Q - f\Q f\Q-\Q f\Q f) = 2(\tr(\Q f^2 \Q) - \tr(\Q f\Q f \Q) ) \\
\tr(B) &=\tr(\Q f^2 \Q) - \tr(\Q f\Q f \Q) 
= -\tfrac12 \tr([f,\Q]^2) = \Sigma,
\end{aligned}\]
this completes the proof.
\end{proof}

\begin{corollary}\label{cor:CLT}
For $\hbar\in(0,1]$, let $\Q_\hbar$ be a family of self-adjoint locally trace-class projections on $L^2(\mathfrak{X}_\hbar)$ and let $\X_\hbar$ denote the associated determinantal point process.
Let $f_\hbar : \mathfrak{X} \to \R$ be a family of functions, with compact support, such that 
$1+\Q_\hbar (e^{\lambda f_\hbar}-1)\Q_\hbar \ge \delta > 0$ for $\lambda\in[0,1]$ and as $\hbar\to0$, 
\[
\Sigma(f_\hbar;\Q_\hbar) \to \sigma^2 , \qquad \sup_{\lambda\in[0,1]} \| [e^{\lambda f}, \Q] \| \to 0. 
\]
Then $\Upsilon(f_\hbar;\Q_\hbar) = \Sigma/2 +\underset{\hbar\to0}{o(1)} $ and, in distribution $(\mathcal{G}$ denotes the standard Gaussian distribution on $\R)$ as $\hbar\to0$, 
\[
\X_\hbar(g_\hbar)- \E \X_\hbar(g_\hbar)  \to  \sigma\mathcal{G}.
\]
\end{corollary}

\begin{proof}
By assumptions, the operators $\Q_\hbar (e^{\lambda f_\hbar}-1)\Q_\hbar$  are trace-class for 
for $\lambda\in[0,1]$ and $\hbar\in(0,1]$ and 
$ \eta(f_\hbar;\Q_\hbar) =\underset{\hbar\to0}{o(1)}$. 
So, applying Proposition~\ref{prop:Gauss}, we obtain 
\[
\sup_{\lambda\in[-1,1]} |\Upsilon( \lambda f_\hbar;\Q_\hbar) - \lambda^2\sigma^2/2 | =\underset{\hbar\to0}{o(1)} .
\]
Equivalently, the Laplace transform $ \E \exp\big(\lambda\X_\hbar(f_\hbar) \big) = \exp\big(\lambda \E \X_\hbar(f_\hbar) + \lambda^2\sigma^2/2 + \underset{\hbar\to0}{o(1)} \big) $ uniformly for $\lambda\in[-1,1]$. 
This yields the central limit theorem.  
\end{proof}

\subsection{CLT for Berezin--Toeplitz ensembles.}
We can directly apply the above results to the Ginibre and other Berezin--Toeplitz ensembles inside the bulk. In these cases, the semi-classical parameter is $\hbar = N^{-1}$. 
Proposition~\ref{prop:cltmeso} is a classical result, but our proof includes the convergence of the Laplace transform and it applies at arbitrary mesoscopic scales.  
Recall that the $\infty$-Ginibre point process has correlation kernel~$P_N$.

\begin{proposition}[Mesoscopic CLT for $\infty$-Ginibre ensemble] \label{prop:cltmeso}
Let $f \in C^1_0(\C)$ be a fixed function and let $g=f(\cdot \eta_N^{-1})$ where $\eta_N\le 1$ satisfies $N\eta_N^2 \to\infty$ as $N\to\infty$. 
Then $\Upsilon(g, P_N) \to  \Sigma^2_\C(f)/2$ as $N\to\infty$. 
In particular, if $\X_N$ denotes $\infty$-Ginibre point process, then as $N\to\infty$
\[
\X_N(g) -  N\int g \d \bgamma \,\Longrightarrow\, \sqrt{\Sigma^2_\C(f)}\, \mathcal{G}. 
\]
This generalizes \eqref{bulkCLT} on mesoscopic scales. 
\end{proposition}

\begin{proof}
Let $C =\|f\|$. 
One can directly apply Corollary~\ref{cor:CLT}. Indeed, according to Lemma~\ref{lem:semiclassic_com}, for $\lambda\in [-1,1]$,
\[
\| [e^{\lambda g}, P_N] \| \lesssim N^{-1/2} e^C \|g\|_{C^1} \lesssim  \|f\|_{C^1}/\sqrt{N \eta_N^2}
\]
so that $\sup_{\lambda\in[0,1]} \| [e^{\lambda g}, P_N] \| \to 0$ as $N\to\infty$. 
Moreover, since $\| [P_N,g] \|_{\rm H}^2 \to \Sigma^2_\C(f) $ as $N\to\infty$, the limit variance is $\sigma^2= \Sigma^2_\C(f) $. 
Observe that the $\infty$-Ginibre point process is stationary with density $P_N(z,z)=N$ for $z\in\C$ so that we can replace $\E[\X_N(g)] = N\bgamma(g)$ in the CLT. 
\end{proof}

Based on the approximations from Proposition~\ref{prop:decay_kernel}, we can generalize the previous proof to  Berezin--Toeplitz ensembles for test functions supported inside the bulk.

\begin{theorem}[Mesoscopic CLT for Berezin--Toeplitz ensembles] \label{prop:cltbulk}
Let $V:\C\to\R$ satisfies Assumptions~\ref{ass:V}  and $\Pi_N = \1\{P_NVP_N\le \mu\}$. 
Let $f \in C^1_0(\{V<\mu\})$ be a fixed function supported on $B(0,1)$,  and let $g=f(\cdot
\eta_N^{-1}+x_N)$ where $\eta_N\le 1$ satisfies  $N\eta_N^2 \to\infty$ as $N\to\infty$, and where
$\dist(x_N,\{V>\mu\})\geq 2 \eta_N$.

Then $\Upsilon(g, \Pi_N) \to  \Sigma^2_\C(f)/2$ as $N\to\infty$. 
In particular, if $\X_N$ denotes the determinantal point process
associated with the projection $\Pi_N$, then as $N\to\infty$,
\[
\X_N(g) - N\int g \d \bgamma \,\Longrightarrow\, \sqrt{\Sigma^2_\C(f)}\, \mathcal{G}. 
\]
\end{theorem}

\begin{proof}
In operator norm, one has
\[
\| [e^{\lambda g}, \Pi_N] \| \le  \| [e^{\lambda g}, P_N] \| + 2 \|(e^{\lambda g}-1) \underbrace{(P_N-\Pi_N)}_{A_N}\|
\]
and $ \|(e^{\lambda g}-1) A_N \| \le  \|(e^{\lambda g}-1) A_N \|_{\rm
  H}$. For the Hilbert-Schmidt norm,  by \eqref{eq:Pi_close_P_bulk}
and using Assumptions \ref{ass:V} b) c), for every $\alpha>0$,
\[
\|(e^{\lambda g}-1)(P_N-\Pi_N)\|_{\rm H}^2 \lesssim \int_{\{g\neq 0\}^2} | A(x,z)|^2 \gamma(\d x)  \gamma(\d z) =O( e^{-\alpha\eta_N\sqrt{N}})
\]
where the implied constant depends on $\|g\| =\|f\|$, $\alpha$, and $V$.

Consequently, as in the proof of Proposition~\ref{prop:cltmeso}, it holds unformly for $\lambda\in [-1,1]$,
\[\begin{aligned}
\| [e^{\lambda g}, \Pi_N] \| \lesssim1/\sqrt{N \eta_N^2}
\qquad\text{and}\qquad 
\| [\Pi_N,g] \|_{\rm H}^2 & = \| [P_N,g] \|_{\rm H}^2 +\O(e^{-\alpha\eta_N\sqrt{N}/2}) \\
&\to  \Sigma^2_\C(f) \quad\text{ as $N\to\infty$} .
\end{aligned}\]
This proves the claims (we can also replace $\Pi_N$ by $P_N$ when computing $\E\X_N(g)$ up to exponentially small errors). 
\end{proof}

\section{Conclusion: Proof of Theorem~\ref{thm:CLT}}

Let $f: \C \to \R$ be bounded and $C^\infty$ in a neighborhood of $\mathcal{D}=\{V<\mu\}$.
The first step towards the proof of Theorem \ref{thm:CLT} is to apply the decorrelation estimate from Corollary~\ref{cor:deco}:
\[\begin{aligned}
\Upsilon(f;\Pi_N) = \Upsilon({\rm f}_1;\Pi_N)+ \Upsilon({\rm f}_2;\Pi_N) - \Upsilon({\rm f}_3;\Pi_N) + \underset{N\to\infty}{o(1)} , \qquad {\rm f}_j:= f \chi_j\text{ for $j\in\{1,2,3\}$}. 
\end{aligned}\]
The cutoffs $\chi_j$ are smooth, $\chi_1$ and $\chi_3$ are  supported in the bulk $\{V\le \mu-\tfrac12\delta\} $, while $\chi_2$ is supported in a neighborhood of the boundary $\{|V- \mu| \le 3\delta\}$. 
in this case ${\rm f}_j \in C^\infty_0(\C)$.

Then, by Proposition~\ref{prop:cltbulk} (with $\epsilon_N=1$), one has for $j\in\{1,3\}$ as $N\to\infty$, 
\[
\Upsilon({\rm f}_j;\Pi_N) \to \tfrac12\Sigma^2_\C({\rm f}_j) = \tfrac12\Sigma^2_{\mathcal{D}}({\rm f}_j). 
\]
Moreover, since ${\rm f}_2$ is supported in a neighborhood of the
boundary, we have shown in  Section \ref{sec:edge} that as $N\to\infty$, 
\[
\Upsilon({\rm f}_2,\Pi_N)  \to  \Sigma^1_{\mathcal{D}}({\rm f}_2)+\Sigma^2_{\mathcal{D}}({\rm f}_2) . 
\]

Altogether, this implies that as $N\to\infty$
\[
\Upsilon(f;\Pi_N) \to \Sigma^1_{\mathcal{D}}({\rm f}_2)+\Sigma^2_{\mathcal{D}}({\rm f}_1) +\Sigma^2_{\mathcal{D}}({\rm f}_2) - \Sigma^2_{\mathcal{D}}({\rm f}_3) .
\]
Obviously $\Sigma^1_{\mathcal{D}}(f) = \Sigma^1_{\mathcal{D}}({\rm f}_2)$ and it remains to check that 
\begin{equation} \label{var2}
\Sigma^2_{\mathcal{D}}(f)=\Sigma^2_{\mathcal{D}}({\rm f}_1) +\Sigma^2_{\mathcal{D}}({\rm f}_2) - \Sigma^2_{\mathcal{D}}({\rm f}_3) . 
\end{equation}

In the droplet $\mathcal{D}$,  ${\rm f}_3={\rm f}_1+{\rm f}_2-f$ (see
the proof of Corollary~\ref{cor:deco}), so that
\[
|\nabla{\rm f}_3|^2 = |\nabla{\rm f}_1|^2 +  |\nabla {\rm f}_2|^2  +|\nabla f|^2 -2 \nabla f \nabla {\rm f}_1 - 2 \nabla f \nabla {\rm f}_2 + 2 \nabla {\rm f}_1\nabla {\rm f}_2
\]
Rearranging and writing ${\rm f}_j = f\chi_j$, we obtain on $\mathcal{D}$,  
\[\begin{aligned}
|\nabla{\rm f}_1|^2 +  |\nabla {\rm f}_2|^2   - |\nabla{\rm f}_3|^2  & = 2 \nabla f \nabla {\rm f}_1 + 2 \nabla f \nabla {\rm f}_2 -  2 \nabla {\rm f}_1\nabla {\rm f}_2
-  |\nabla f|^2 \\
&=  2 \nabla f \nabla \chi_1 + 2 \nabla f \nabla \chi_2  -  2 \nabla {\rm f}_1\nabla {\rm f}_2 +3  |\nabla f|^2 \\
&=  |\nabla f|^2
\end{aligned}\]
where we used that 
\[ 
\nabla {\rm f}_1\nabla {\rm f}_2 = | \nabla f |^2+  \nabla \chi_1\nabla f +  \nabla f \nabla \chi_2  
\]
since $\nabla \chi_1 \nabla \chi_2 =0$.
Integrating this formula over $\mathcal{D}$,  this proves
\eqref{var2}.

\bibliographystyle{plain}
\bibliography{Ginibre}

\end{document}